\documentclass[11pt]{article}

\newcommand{\N}{\mathbb{N}}
\newcommand{\Leq}{\preccurlyeq}
\newcommand{\SetIdeals}{\mathrm{J}}
\newcommand{\SetJoinIrreducibles}{\mathrm{JIrr}}
\newcommand{\SetNLT}{\mathrm{NLT}}
\newcommand{\Grow}{\mathrm{grow}}
\newcommand{\First}{\mathrm{first}}
\newcommand{\Next}{\mathrm{next}}
\newcommand{\Cut}{\mathrm{cut}}

\usepackage[hidelinks]{hyperref}
\usepackage{amsmath}
\usepackage{eqnarray,amsmath}
\usepackage{amssymb}
\usepackage{amsfonts}
\usepackage{tikz-cd}
\usepackage[usenames,dvipsnames]{pstricks}
\usepackage{graphicx}
\usepackage{subcaption}
\usepackage[labelformat=parens,labelsep=quad,skip=3pt]{caption}
\usepackage{qtree}
\usepackage{algorithm}
\usepackage[noend]{algpseudocode}
\usepackage[utf8]{inputenc}
\usepackage{amsthm}
\usepackage{pstricks-add,}
\usepackage{mathtools}
\usepackage{dirtytalk}
\usepackage{cleveref}
\usepackage{cancel}

\usepackage{xcolor}
\usepackage{hyperref}

\usepackage[english]{babel} % needed for "blindtext"
\usepackage[pangram]{blindtext}

\usepackage{array}
\usepackage{rotating}
\usepackage{tikz}
\usepackage{blkarray}

\usepackage{MnSymbol} % It is for \udots command

\usepackage{bbm}

\usepackage{nicematrix}
\NiceMatrixOptions{
	code-for-first-row = \color{blue}, 
	code-for-last-col = \color{blue},  
	code-for-first-row = \color{blue}, 
	code-for-last-col = \color{blue}, 
}

\hypersetup{colorlinks=false,linkbordercolor=red,linkcolor=green,pdfborderstyle={/S/U/W 1}}

\usepackage{atbegshi}% http://ctan.org/pkg/atbegshi
\AtBeginDocument{\AtBeginShipoutNext{\AtBeginShipoutDiscard}}

\newtheorem{theorem}{Theorem}[section]

\newtheorem{proposition}[theorem]{Proposition}

\newtheorem{remark}[theorem]{Remark}
\newtheorem{corollary}[theorem]{Corollary}
\newtheorem{example}[theorem]{Example}

\tikzset{use/.style={}}
\tikzset{rectangle/.append style={circle, fill=blue!25}}

\newcommand{\PosetInsertion}[3]{
    \ensuremath{
    \mathop{
    \begin{tikzpicture}[scale=0.25]
        \ifnum#1=1 \fill[black!75] (0,0.5) rectangle (0.5,1); \fi
        \ifnum#2=1 \fill[black!75] (0,0) rectangle (0.5,0.5); \fi
        \ifnum#3=1 \fill[black!75] (0.5,0) rectangle (1,0.5); \fi
        \draw[thin](0,0.5) rectangle (0.5,1);
        \draw[thin](0,0) rectangle (0.5,0.5);
        \draw[thin](0.5,0) rectangle (1,0.5);
    \end{tikzpicture}
    }}
}

\newcommand{\PosetComposition}[2]{
    \,
    \begin{tikzpicture}[scale=0.25]
        \ifnum#1=1 \fill[black!75] (0,0.5) rectangle (0.5,1); \fi
        \ifnum#2=1 \fill[black!75] (0.5,0) rectangle (1,0.5); \fi
        \draw[thin](0,0.5) rectangle (0.5,1);
        \draw[thin](0.5,0) rectangle (1,0.5);
    \end{tikzpicture}
    \,
}

\begin{document}

\title{Generating naturally labeled posets through \\matrix extensions, order ideals and automorphism groups}
\author{Gi-Sang Cheon$^{a}$\thanks{Corresponding author}, Samuele Giraudo$^{b}$, Gukwon Kwon$^{a}$, Hojoon Lee$^{a}$\\
{\footnotesize $^a$ \textit{Department of Mathematics, Sungkyunkwan
University (SKKU), Suwon 16419, Rep. of Korea}}\\
{\footnotesize $^{b}$ \textit{Universit\'e du Qu\'ebec \`a Montr\'eal (UQAM), LACIM, Montr\'eal, H2X 3Y7, Canada
}}\\
{\footnotesize gscheon@skku.edu, giraudo.samuele@uqam.ca, tomy1995@skku.edu, hojoon1101@g.skku.edu
}}

\thanks{This work was supported by the National Research Foundation of Korea(NRF) Grant funded by the Korean Government (MSIT)(RS-2025-00573047).
S. Giraudo was also partially supported by the Natural Sciences and Engineering Research
Council of Canada (RGPIN-2024-04465).}
 
\date{}
\maketitle

\begin{abstract}
We propose a matrix approach for generating naturally labeled posets by representing each poset $P$ on the set $[n]$ as a Boolean poset matrix $A$. 
This algebraic representation enables a systematic handling of partial orderings through matrix extensions of the form $A^v=\bigl[\begin{smallmatrix}A&0\\ v&1\end{smallmatrix}\bigr]$.
We show that $A^v$ defines a valid poset matrix if and only if the Boolean vector $v\in\mathbb B^n$ represents an order ideal of the poset $P$ associated to $A$, equivalently satisfying the fixed-point equation $vA=v$. Based on this characterization, we develop a sieve algorithm that generates all admissible extension vectors efficiently.
Furthermore, we explore the twin-class decomposition of $A$, which partitions the elements of $P$ according to identical down- and up-sets. 
This structure provides an algebraic foundation for Burnside-type enumeration for Birkhoff’s question on counting nonisomorphic posets on $[n]$ through the automorphism group ${\rm Aut}(A)$. 
Finally, we present an algorithmic generation scheme for the posets based on the topological growth of their distributive lattices, offering a new approach to constructive enumeration of poset families.
\end{abstract}



\section{Introduction}
Let $X_n$ denote the $n$-set $\{0,1,\ldots,n-1\}$. A {\it partial order} $\preceq$ on $X_n$ is a binary relation that is reflexive, antisymmetric, and transitive. A partially ordered set (shortly, {\it poset}) is an ordered pair $P = (X_n, \preceq)$ where we write $x \prec y$ if $x \preceq y$ and $x \ne y$. Two posets $P$ and $Q$ on $X_n$ are said to be {\it isomorphic}, denoted by $P\cong Q$, if there exists a bijection $\theta: P \rightarrow Q$ such that $x \preceq y$ if and only if $\theta(x) \preceq \theta(y)$. We denote ${\mathcal{NIP}}(n)$ by the collection of all {\it non-isomorphic} posets (shortly, NIPs) with $n$ elements. A classic problem, originally posed by Birkhoff \cite{Birkhoff}, is to determine $|{\mathcal{NIP}}(n)|$. This combinatorial challenge has remained unsolved, with known \cite{Brinkmann} only up to $n=16$. 
 
A partial order $\preceq$ on $X_n$ is said to be {\it natural} if $x \preceq y$ implies $x
\le y$. We refer to such posets as \emph{naturally labeled} posets (shortly, NL posets)
\cite{Bevan, Dean}. As pointed in \cite{Dean}, every finite poset is isomorphic to an NL poset and
every NL poset has a matrix representation by its order relation. From this point of
view, we are interested in enumerating NL posets using their matrix representations. The collection of all NL posets on $X_n$ will be denoted by
$\mathcal{NL}(n)$. The counting sequence for NL posets on $X_n$ is known up to $n=12$ (see
A006455 in the OEIS \cite{OEIS}).  

The incidence matrix $A$ of a poset $P\in \mathcal{NL}(n)$ with labeling $0,1,\ldots,n-1$ is defined by
\begin{eqnarray*}\label{pm}
A_{ij} = \begin{cases}
1 & \text{if } j \preceq i, \\
0 & \text{otherwise}
\end{cases}
\end{eqnarray*}
where $A_{ij}$ denotes the $(i,j)$-entry of $A$, and the indices $i,j$ start from $0$. We refer to the $n\times n$ matrix $A$ as the {\it poset matrix} of a poset $P\in \mathcal{NL}(n)$ (see also \cite{Cheon1, {CKLW},  Moha}). We will sometimes write $A_P$ instead of $A$. By definition, $A$ is a poset matrix if and only if $A$ is a $(0,1)$-lower triangular matrix with 1s on the main diagonal and it is transitive {\it i.e.}, $A_{ij}=1$ and $A_{jk}=1$ then $A_{ik}=1$; equivalently, $A$ has no submatrices of the form $(\begin{smallmatrix}1&{\bf 1}\\0&1\end{smallmatrix})$ where the bold entry is on the main diagonal (\cite{Bevan}). Accordingly,
if $A$ is a poset matrix then there exists a unique NL poset $P$ whose poset matrix is $A$, and we write $P_A$ for $P$. Let $\mathcal{PM}(n)$ denote the collection of all $n\times n$ poset matrices.
\begin{proposition}\label{pro1} (\cite{Bevan})
 There is a bijection between $\mathcal{NL}(n)$ and $\mathcal{PM}(n)$. 
\end{proposition}

By Proposition \ref{pro1}, the enumeration of NL posets on $X_n$ is equivalent to the enumeration of poset matrices in $\mathcal{PM}(n)$.
It is important to note that an unlabeled poset $P=(X_n,\preceq)$ may admit different
natural labeling, each inducing a distinct poset matrix.  Two poset matrices $A$ and $B$ in
${\cal PM}(n)$ are isomorphic, denoted by $A\cong B$, if there exists a permutation matrix
$Q\in \mathfrak{P}_n$ such that $$A=Q^TBQ$$ where $\mathfrak{P}_n$ is the group of $n\times
n$ permutation matrices. It is easy to show that if $A\cong B$ then $P_A\cong P_B$, and both
are isomorphic to the unlabeled poset obtained from deleting their labels. It follows that
the enumeration problem on finite posets can be considered on NL posets up
to isomorphism.

In this paper, we propose a matrix approach to the enumeration and generation of NL posets by observing that
every NL poset $P$ on $X_{n+1}$ can be obtained from some NL poset on $X_n$ by adjoining a new element together with its order relation.
In the corresponding matrix representation, this process is naturally expressed by
\begin{eqnarray*}\label{def}
\left[\begin{array}{cc} 
A & {\bf 0}\\
v & 1
\end{array}\right]
:=A^v,
\end{eqnarray*}
where $A\in\mathcal{PM}(n)$ is the poset matrix of the subposet on $X_n$, and $v\in\{0,1\}^n$ encodes the order relation between the new element $n$ and the elements $0,\ldots,n-1$ in $X_n$.
However, given a poset matrix $A\in\mathcal{PM}(n)$, not every $v\in\{0,1\}^n$ yields a valid poset matrix $A^v$ in $\mathcal{PM}(n+1)$.
Characterizing those vectors $v$ that preserve the transitivity of $A^v$ is the key to understanding how NL posets grow from $A$.
We call such a vector $v\in\{0,1\}^n$ a {\it poset vector} of $A$ if $A^v\in\mathcal{PM}(n+1)$. In that case, the resulting matrix $A^v$ is called a {\it matrix extension} of $A$, or simply a {\it $v$-extension} of $A$. Determining all admissible poset vectors of a given $A$ thus provides a matrix-based method for enumerating and generating NL posets, and it also connects with Birkhoff's question on NIPs.

More specifically, Section~2 characterizes all admissible vectors $v\in\{0,1\}^n$ by showing that $A^v$ defines a valid poset matrix if and only if the Boolean vector $v$ represents an order ideal of the poset $P_A$, equivalently satisfying the {\it fixed-point equation} $vA=v$ in Boolean matrix product. In Section 3, we develop a sieve algorithm for generating all admissible poset vectors associated with a given poset matrix. Section~4 investigates algebraic and combinatorial properties of the poset vector set, $\mathsf{PV}_n(A)$
and explains how this set affects the structure of $v$-extensions of $A$. In Section~5, we study the automorphism group
${\rm Aut}(P)$ using the twin-class decomposition of $P:=P_A$, which partitions the elements of $P$ according to identical
strict down- and up-sets. This decomposition separates the symmetries of a poset into internal symmetries within twin classes and external symmetries coming
from automorphisms of the associated twin poset. As a consequence, we obtain the semidirect product decomposition (see Theorem \ref{aut}),
$$
{\rm Aut}(P)
\cong
\left(\prod_{i=1}^r \mathfrak S_{|\alpha_i|}\right)\rtimes {\rm Im}(\pi),
$$
where $\alpha_1,\ldots,\alpha_r$ are the twin classes of $P$.

 Finally, in Section 6, we present an algorithmic generation scheme for NL posets based on the topological growth of their distributive lattices, offering a new approach to constructive enumeration of poset families.

\section{Admissible poset vectors}

A {\it Boolean vector} of size $n$ is an $n$-tuple of elements from the {\it Boolean algebra} $\mathbb{B}=\{0,1\}$ with the operations $+$ and $\cdot$ defined as follows \cite{Kkim}: 
 $$
 0+0=0\cdot 1=1\cdot0=0\cdot0=0,\quad\text{and}\quad 1+0=0+1=1+1=1\cdot1=1.
 $$
 By convention, we denote the zero vector by ${\bf 0}$. Let $\mathbb{B}^n$ denote the Boolean vector space of dimension $n$ equipped with entrywise addition.
 A $(0,1)$-matrix over the Boolean algebra $\mathbb B$ is called a {\it Boolean matrix}. 
  Throughout this paper, poset matrices and poset vectors are assumed to be Boolean matrices. We also consider only natural partial orderings.
 It is known~\cite{kim} that if $A$ is a poset matrix, it is idempotent over $\mathbb{B}$, that is, $A^2 = A$.   

Let $r_0,\dots,r_{n-1}\in \mathbb{B}^n$ denote the row vectors of $A=[a_{ij}]\in{\mathcal PM}(n)$. We define the {\it support} of $r_i$ by
$$
{\rm supp}(r_i):=\{j\in[n]\mid a_{ij}=1\}.
$$
A subset $J$ of $X_n$ is called an {\it order ideal} of $P=(X_n,\preceq)$ if $x\in J$ and $y\preceq x$ imply $y\in J$, i.e., $J$ is {\it downward closed} (see \cite{Davey}).
 The {\it principal} order ideal of $P$ generated by $x\in P$ means a {\it down set} of the form
$$
\downarrow x:=\{y\in P\mid y \preceq x\}.%,\quad \uparrow x :=\{y\in P\mid x \preceq y\}.
$$
By definition, ${\rm supp}(r_i)$ coincides with the principal order ideal $\downarrow i$ of $P$. Thus each $r_i$ can be considered as the incidence vector of $\downarrow i$.

\begin{theorem}\label{ppp}
Let $A\in\mathcal{PM}(n)$ and $v\in{\mathbb B}^n$. 
Then $v$ is a poset vector of $A$ if and only if $\mathrm{supp}(v)$ is an order ideal of $P_A$.
Moreover, the map $\mathrm{supp}$ is a bijection between $v$-extensions of $A$ and order ideals of $P_A$.
\end{theorem}

\begin{proof}
Let $v\in{\mathbb B}^n$ be a poset vector of $A\in\mathcal{PM}(n)$. Then $A^v\in\mathcal{PM}(n+1)$ by definition. 
Let $A^v=[a_{i,j}]$ where $v=(a_{n,0},\ldots,a_{n,n-1}).$ Since $A^v$ is transitive, it follows that $a_{n,i}=1$ and $a_{i,j}=1$ implies $a_{n,j}=1$. Equivalently, if $i\in{\rm supp}(v)$ and $j\preceq i$ in $P_A$, then $j\in{\rm supp}(v)$. Hence ${\rm supp}(v)$ is an order ideal of $P_A$. Conversely, if ${\rm supp}(v)$ is an order ideal of $P_A$, then
the above implication holds for all $i,j\in X_n$. Thus $A^v$ satisfies transitivity, which implies $A^v\in\mathcal{PM}(n+1)$ and hence $v$ is a poset vector of $A$.

Moreover, since each poset vector $v$ determines a unique $v$-extension $A^v$ 
and vice versa, the correspondence $v\leftrightarrow\ {\rm supp}(v)$ gives a bijection between the family of $v$-extensions of $A$ and the set of order ideals of $P_A$. 
This completes the proof.
\end{proof}

In terms of Theorem \ref{ppp}, a poset vector $v$ of $A$ is referred to as a {\it
characteristic vector} of the order ideal $J$ of $P_A$, and we write $v=\chi(J)$. For
a subset $S\subseteq X_n$ with $|S|=k$, we denote by $A[S]$ the $k\times k$ principal
submatrix of $A$ obtained by taking the rows and columns whose indices lie in $S$. Recall
that the Boolean vector sum is the entrywise sum over the Boolean algebra $\{0,1\}$.

\begin{theorem}\label{thm333}
Let $A\in\mathcal{PM}(n)$ be a poset matrix with row vectors 
$r_0,\ldots,r_{n-1}\in{\mathbb B}^n$. For $v\in {\mathbb B}^n$, $v$ is a poset vector of $A$ if and only if $v=\sum_{i\in S} r_i$ for some $S\subseteq X_n$ with $|S|=k\ge0$ such that 
$A[S]$ is the $k\times k$ identity matrix $I_k$, where the sum is the Boolean vector sum and $v=\mathbf{0}$ when $k=0$.
\end{theorem}
\begin{proof} Assume  $v\in {\mathbb B}^n$ is a poset vector of $A$. By Theorem~\ref{ppp}, 
$J:=\mathrm{supp}(v)$ is an order ideal of $P_A$. If $J=\emptyset$, then $v=\mathbf{0}$ and thus the stated property holds with $S=\emptyset$.
Suppose $J\neq\emptyset$ and let $S$ be the set of maximal elements of $J$ where $|S|=k$. We show that $S$ is the subset of $X_n$ with the required properties. Clearly, $S\subseteq X_n$. Since maximal elements are pairwise incomparable, it follows $A[S]=I_k$. Moreover, since every element of $J$ lies in the down-set of some $i\in S$, we have
$$
\mathrm{supp}(v)=J=\bigcup_{i\in S}\downarrow i=\bigcup_{i\in S}\mathrm{supp}\;(r_i)=\mathrm{supp}\;\!\Big(\sum_{i\in S} r_i\Big).
$$
The maximality of $|S|$ implies $v=\sum_{i\in S}r_i$, as required. 

Conversely, suppose $S$ is a subset of $X_n$ with $A[S]=I_k$ such that $v=\sum_{i\in S} r_i$. Then
$$
\mathrm{supp}(v)=\mathrm{supp}\Big(\sum_{i\in S} r_i\Big)\;=\;\bigcup_{i\in S}\downarrow i,
$$
which is a union of principal ideals of $P_A$. Hence $\mathrm{supp}(v)$ is an order ideal of $P_A$. It follows from Theorem~\ref{ppp} that
$v$ is a poset vector of $A$. The case $S=\emptyset$ yields $v=\mathbf{0}$ and is trivial. This completes the proof.
\end{proof}

By Theorems \ref{ppp} and \ref{thm333}, each order ideal $J$ of $P_A$ is uniquely generated by its
antichain of maximal elements, and conversely, each antichain $S$ generates the order ideal
$$\downarrow {S}=\{x\in P\mid x\preceq y\;{\text{for all $y\in S$}}\}.
$$
of $P_A$ (also see \cite{Davey, Stanley}).
\medskip

The following corollary ensures that $v$-extensions preserve the transitivity and partial order by enforcing the order ideal condition. It also reduces the search for all possible extensions $A^v$ to the enumeration of antichains in the NL poset $P_A$.

\begin{corollary}\label{ccol}
Let $v$ be a poset vector of $A\in\mathcal{PM}(n)$, and let $P^v:= P_A\cup\{n\}$ be the NL poset on $X_{n+1}$ associated to $A^v\in\mathcal{PM}(n+1)$.
Then the new element $n$ covers $x\in P_A$ if and only if $x\in S$, where $S$ is an antichain of $P_A$, including the empty antichain.
\end{corollary}

Recently, the authors \cite{CKLW} showed that counting antichains of the {\it Pascal poset} ${\mathbb P}_n$ is equivalent to enumerating order ideals of ${\mathbb P}_n$. 
Moreover, they established a correspondence between order ideals of $\mathbb P_n$ 
and Boolean vectors $x \in \mathbb B^n$ satisfying the {\it Boolean fixed point equation}
$x {\bf P}_n = x$, where ${\bf P}_n = [p_{i,j}]$ is the $n \times n$ binary Pascal matrix defined by
\begin{equation}\label{Lucas}
p_{i,j} \;=\; \binom{i}{j} \pmod{2},
\end{equation}
which is the poset matrix of the Pascal poset ${\mathbb P}_n$ (see \cite{Cheon1}).

Recall that the {\it Dedekind number} \cite{Dedekind}, denoted by $M(n)$, is the number of
monotone Boolean functions in $n$ variables. Equivalently, it is the number of
antichains in the $n$-dimensional {\it Boolean lattice} 
$\mathbb{B}_n = (2^{X_n}, \subseteq)$, consisting of all subsets of $X_n$ ordered by inclusion. 
Exact values of $M(n)$ are known only for $0 \le n \le 9$ (A000372 in the OEIS):
$$
M(0)=2,\; M(1)=3,\; M(2)=6,\; M(3)=20,\; M(4)=168,\; M(5)=7581,\; M(6)=7828354, \ldots.
$$

Since ${\bf P}_{2^n}$ is the poset matrix of the Boolean lattice $\mathbb B_n$, 
which is isomorphic to the Pascal poset ${\mathbb P}_{2^n}$, 
the number of its poset vectors coincides with the number of antichains of $\mathbb B_n$. 
Thus we obtain the following result.

\begin{theorem}\label{Bpascal}
The number of poset vectors of ${\bf P}_{2^n}$ for $n \ge 0$ is equal to the Dedekind number $M(n)$.
\end{theorem}
 
    For example, let ${\bf P}_{2^n}^{v}$ be $v$-extensions with $v\in{\mathbb B}^{2^n}$. Consider $n=0,1,2$:
{\small\begin{itemize}
\item   ${\bf P}_1^{v}=\left[
\begin{array}{c|c}
1 & 0 \\ \hline
v & 1
\end{array}
\right]
\in\mathcal{PM}(2)\Leftrightarrow v\in\{[0],[1]\}$,\quad $M(0)=2$;
\item   ${\bf P}_2^{v}=
\left[
\begin{array}{cc|c}
1 & 0 & 0 \\
1 & 1 & 0 \\ \hline
\multicolumn{2}{c|}{v} & 1
\end{array}
\right]
\in\mathcal{PM}(3)\Leftrightarrow v\in\{[0\;0],[1\;0],[1\;1]\}$,\quad $M(1)=3$;
 \item  $ {\bf P}_4^{v}= 
\left[
\begin{array}{cccc|c}
1 & 0 & 0 & 0 & 0 \\
1 & 1 & 0 & 0 & 0 \\
1 & 0 & 1 & 0 & 0 \\
1 & 1 & 1 & 1 & 0 \\ \hline
\multicolumn{4}{c|}{v} & 1
\end{array}
\right]
 \in\mathcal{PM}(5)\Leftrightarrow v\in\{[0\;0\;0\;0],[1\;0\;0\;0],[1\;1\;0\;0],[1\;0\;1\;0],[1\;1\;1\;1],[1\;1\;1\;0]\}$. Hence $M(2)=6$.
 \end{itemize}}
 \medskip
 
  Denote by $\SetIdeals(P)$ the set of order ideals of $P\in{\cal NL}(n)$, by
$\textsf{PV}_n(A)$ the set of poset vectors of $A\in\mathcal{PM}(n)$, and by $\textsf{NL}_{n+1}(P_A^v)$ the set
of NL posets on $X_{n+1}$ associated to any $v\in\textsf{PV}_n(A)$.
\medskip

One of the most important ways to study a Boolean matrix is to consider its row space. It is known \cite{Kkim} that the row space forms a {\it lattice}, and conversely every lattice can be represented as the row space of some Boolean matrix. 
Let $\mathsf{Row}_{\mathbb{B}}(A)$ be the row space of the poset matrix $A\in{\mathcal PM}(n)$, which is a subspace of $\mathbb{B}^n$ spanned by its rows.

\begin{theorem}\label{rowspace}
Let $A\in\mathcal{PM}(n)$ be a poset matrix. Then $\mathsf{Row}_{\mathbb{B}}(A)=\mathsf{PV}_n(A).$
\end{theorem}
\begin{proof} Since the row space of $A$ is the set of all finite sums of row vectors of $A$, it follows that
\begin{eqnarray*}
\mathsf{Row}_{\mathbb{B}}(A)
&=&\Bigl\{\;\sum_{i\in S} r_{i}\;\Bigm|\;S\subseteq[n]\Bigr\}
=\Bigl\{\;v\in\mathbb B^n\;\Bigm|\;
\operatorname{supp}(v)\in  \SetIdeals(P_A)\Bigr\}\\
&=&\Big\{\, \sum_{i\in M} r_i \;\Big|\; M\subseteq [n]\text{ is an antichain of }P_A \Big\}\\
&=&\textsf{PV}_n(A),
\end{eqnarray*}
 which completes the proof.
\end{proof}

Since Theorem~\ref{ppp} implies $\textsf{PV}_n(A)\cong\SetIdeals(P_A)$, we have 
$$
|\mathsf{Row}_{\mathbb{B}}(A)|=|\textsf{PV}_n(A)|=|\SetIdeals(P_A)|=|\textsf{NL}_{n+1}(P_A^v)|.
$$
Thus, the {\it total} number of NL posets on $X_{n+1}$ can be recursively obtained by
\begin{eqnarray*}\label{NLP}
|\mathcal{NL}(n+1)|=\sum_{A\in{\cal PM}(n)} |\textsf{PV}_n(A)|=\sum_{A\in{\cal PM}(n)} |\textsf{NL}_{n+1}(P^v_A)|.
\end{eqnarray*}

It is known \cite{Avann, Dean} that $\mathcal{NL}(n)$ is the lattice ordered by $\subseteq$, where $P\subseteq Q$ if $x\preceq_P y$ implies $x\preceq_Q y$ for all $x,y\in X_n$. 
Theorem \ref{rowspace} implies the following corollaries. 

\begin{corollary}
Let $A\in\mathcal{PM}(n)$ be a poset matrix. Then $\mathsf{PV}_n(A)$ is a distributive
lattice ordered by entrywise $\leq$.
Moreover, $\mathsf{NL}_{n+1}(P^v_A)$ is a sublattice of $\mathcal{NL}(n+1)$.
\end{corollary}

\begin{proof} By Theorem \ref{rowspace}, there is a one-to-one correspondence $\Phi:\mathsf{PV}_n(A)\to \SetIdeals(P_A)$ given by $\Phi(v)=\mathrm{supp}(v)$.
In addition, for $u,v\in\{0,1\}^n$, $u\le v$ if and only if $\mathrm{supp}(u)\subseteq \mathrm{supp}(v)$, so that $\Phi$ is an order isomorphism.
Since $\SetIdeals(P_A)$ forms a distributive lattice, it follows that $\mathsf{PV}_n(A)$, being order-isomorphic to $\SetIdeals(P_A)$, is also a distributive lattice. Moreover, $\mathsf{NL}_{n+1}(P^v_A)$ is a sublattice of $\mathcal{NL}(n+1)$.
\end{proof}

\begin{corollary}
The lattice $\mathcal{NL}(n+1)$ is the disjoint union of the sublattices
$\mathsf{NL}_{n+1}(P^v_A)$, that is
$$
\mathcal{NL}(n+1)
\;=\;
\bigsqcup_{A\in\mathcal{PM}(n)} \mathsf{NL}_{n+1}(P^v_A).
$$
\end{corollary}

We end this section by noticing that all poset vectors in $\textsf{PV}_n(A)$ provide different NL posets on $X_{n+1}$ but {\it not all} are non-isomorphic.
 
  For example, let $A=\begin{bmatrix}
1 & 0 \\
0 & 1\end{bmatrix}$. 
Then $\textsf{PV}_2(A)=\{(0,0),(1,0),(0,1),(1,1)\}$. Since
$$
A^{(1,0)}=\begin{bmatrix}
1 & 0 & 0 \\[2pt]
0 & 1 & 0 \\ 
1 & 0 & 1
\end{bmatrix}\cong
\begin{bmatrix}
1 & 0 & 0 \\[2pt]
0 & 1 & 0 \\ 
0 & 1 & 1
\end{bmatrix}=A^{(0,1)},
$$
the associated NL posets are isomorphic. Thus we obtain three 
non-isomorphic posets from the vector extensions of $A$:

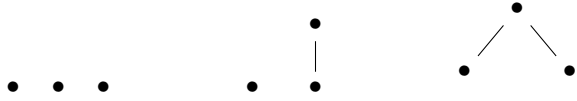
\usetikzlibrary{positioning}
\tikzset{
  pn/.style={circle,inner sep=1.5pt,draw,fill=black},
  newn/.style={circle,inner sep=1.5pt,draw,fill=red!12},
  hasse/.style={line width=0.6pt},          
  redge/.style={line width=1pt,red},        
}
\begin{center}
\begin{tikzpicture}[scale=0.7]
    \node[pn] (x0) at (0,0) {};
    \node[pn] (x1) at (1,0) {};
    \node[pn] (x2) at (2,0) {};

    \node[pn] (y0) at (5,0) {};
    \node[pn] (y1) at (6,0) {};
    \node[pn] (y2) at (6,1) {};
    \draw[hasse] (y1) -- (y2);

    \node[pn] (z0) at (9,0) {};
    \node[pn] (z1) at (11,0) {};
    \node[pn] (z2) at (10,1) {};
    \draw[hasse] (z0) -- (z2);
    \draw[hasse] (z1) -- (z2);
\end{tikzpicture}
\end{center}

This simple example illustrates that although every poset vector in 
$\textsf{PV}_n(A)$ generates a new NL poset on $X_{n+1}$, 
many of these NL posets may be isomorphic. This observation raises several structural, enumerative, and algorithmic questions 
about the $v$-extensions of a poset matrix $A \in \mathcal{PM}(n)$.
\begin{itemize}
\item [Q1.] Can all admissible poset vectors in $\mathsf{PV}_n(A)$ be generated by an efficient algorithm?
  \item [Q2.] For which pairs $v,w\in\mathsf{PV}_n(A)$ are the corresponding 
posets $P_{A^v}$ and $P_{A^w}$ isomorphic?  
Can we describe the action of the automorphism group ${\rm Aut}(A)$ on $\mathsf{PV}_n(A)$ 
so that distinct isomorphism types correspond to its orbits?
  \item[Q3.] How many non-isomorphic $v$-extensions arise from all 
$v\in\mathsf{PV}_n(A)$?  
Equivalently, what is the cardinality of the orbit set 
$\mathsf{PV}_n(A)/{\rm Aut}(A)$?
\item [Q4.] Can the enumeration of $\mathcal{NL}(n+1)$ be recursively obtained 
from that of $\mathcal{NL}(n)$ by summing over all distinct 
isomorphism types of $v$-extensions?
\end{itemize}

The next several sections are devoted to investigating these questions.

\section{Generating poset vectors from a sieve algorithm}

In combinatorial enumeration, it is often more efficient to exclude forbidden configurations than to construct admissible objects directly. 
This idea is the basis of classical sieve methods and the inclusion-exclusion principle. 
Motivated by this viewpoint, we develop a sieve algorithm for generating all admissible poset vectors associated with a given poset matrix.

Let $A\in\mathcal{PM}(n)$ be a poset matrix. By \cite{CKLW}, every poset matrix $A$ appears as an induced principal submatrix of the binary Pascal matrix ${\bf P}_{2^n}$. 
More precisely, there exists a unique set
$$
\alpha=\{\alpha_0,\ldots,\alpha_{n-1}\}\subseteq [0,2^n-1]
$$
such that $A={\bf P}_{2^n}[\alpha]$, the principal submatrix of ${\bf P}_{2^n}$ indexed by $\alpha_0,\ldots,\alpha_{n-1}$,  where the indices satisfy
$$
A(2^0,\ldots,2^{n-1})^T=(1,\alpha_1,\ldots,\alpha_{n-1})^T,
\quad 2^i\le \alpha_i\le 2^{i+1}-1.
$$
We call $\alpha_i$ the {\it Pascal index at level $i$} of $A$, and write $\alpha=\alpha(A)$. 
From $\alpha_i=\sum_{j=0}^{i} A_{ij}2^j$, it follows that $\alpha_i$ encodes the support of $i$th row of $A$, denoted by $S(\alpha_i):=\{j\le i : A_{ij}=1\}$,
which corresponds to the principal ideal $\downarrow i=\{ j\in P_A \mid j\preceq i\}$ in the associated NL poset $P_A$.

Now consider a $v$-extension $A^v$ of $A$, where $v=(v_0,\ldots,v_{n-1})\in\mathbb{B}^n$. Define the {\it binary index} of $v$ by
$$
b(v):=\sum_{i=0}^{n-1} v_i2^i \;=\;\sum_{i\;\in\; \mathrm{supp}(v)}2^i,\quad 0\le b(v)\le 2^n-1.
$$
This establishes a bijection between $\mathbb{B}^n$ and $[0,2^n-1]$. 
Under this identification, the poset vectors of $A$ correspond precisely to those integers $b(v)$ that satisfy the admissibility constraints.
Hence, the problem of determining whether $v$ is a poset vector of $A$ can be reformulated as deciding whether the integer $b(v)$ satisfies certain constraints induced by the Pascal indices  $\alpha=\{\alpha_0,\ldots,\alpha_{n-1}\}$ such that $A={\bf P}_{2^n}[\alpha]$. 
These constraints are determined by the binomial structure of the Pascal matrix. Thus, the compatibility of a candidate index $b(v)$ with the poset structure is determined by binomial coefficients modulo $2$. 
A violation occurs precisely when a coefficient of the form $\binom{j}{\alpha_i}$ vanishes modulo $2$, indicating that the corresponding relation is forbidden.
Therefore, instead of constructing admissible vectors directly, it is natural to encode all such forbidden configurations. 

For this purpose, we introduce the family of matrices
$$
{\cal F}_{2^i}:=\big({\bf P}_{2^i}^T\big)^C
=J-{\bf P}_{2^i}^T,\quad 0\le i\le n-1,
$$
where $J$ is the all-ones matrix. By construction, for all $(r,j)$-elements where $0\le r,j\le 2^i-1$,
$$
({\cal F}_{2^i})_{r,j}=1
\;\Longleftrightarrow\;
\binom{j}{r}\equiv 0 \pmod{2}.
$$
Thus, while ${\bf P}_{2^i}$ encodes admissible relations, its transpose complement ${\cal F}_{2^i}$ determines exactly the forbidden configurations that must be excluded in the sieve process.
For example,
$$
{\cal F}_{2^0}=[0],\quad {\cal F}_{2^1}=\left[\begin{array}{cc} 0&0\\
1&0
\end{array}\right],\quad {\cal F}_{2^2}=\left[\begin{array}{cc|cc} 0&0&0&0\\1&0&1&0\\
\hline 1&1&0&0\\1&1&1&0\end{array}\right].
$$

\medskip

Forming the block diagonal matrix $\mathcal F$ with row and column indices $1,\ldots,2^n-1$ defined by
\[
\mathcal{F}:={\cal F}_{2^0}\oplus{\cal F}_{2^1}\oplus\cdots\oplus{\cal F}_{2^{n-1}},
\]
we obtain a unified representation of forbidden configurations across all levels $i=0,\ldots,n-1$ for the Pascal index $\alpha_i\in [2^i,2^{i+1}-1]$. Denote by $Y_r$ the support of $r$th row of $\mathcal{F}$, that is,
$$
Y_r=\{ j\in[1,2^n-1]\mid \mathcal{F}_{r,j}=1\},\quad r=1,\ldots,2^n-1.
$$
These sets encode exactly the forbidden residue classes that must be removed in the sieve algorithm.

\begin{proposition}\label{YY} Let \(A\in\mathcal{PM}(n)\) and let \(\alpha(A)=\{\alpha_0,\alpha_1,\ldots,\alpha_{n-1}\}\) be its Pascal index set.
Then $Y_{\alpha_0}=\emptyset$ and for $i\ge1$,
\begin{eqnarray}\label{Y}
Y_{\alpha_i}=\left\{j\in [2^i,2^{i+1}-1]\mid {j\choose \alpha_i}\equiv 0\; ({\rm mod}\;2)\right\}.
\end{eqnarray}
\end{proposition}
\begin{proof} Since $\alpha_0=1$, clearly $Y_{\alpha_0}=\emptyset$. Let $i\ge1$. Since $\alpha_i \in [2^i,2^{i+1}-1]$, it follows that $Y_{\alpha_i}\subseteq [2^i,2^{i+1}-1]$. Using (\ref{Lucas}) we have
$$
\mathcal F_{\alpha_i,j}
=1-({\bf P}_{2^i}^T)_{\alpha_i,j}
=1-\binom{j}{\alpha_i}\ (\mathrm{mod}\ 2),\quad j\in [2^i,2^{i+1}-1].
$$
Thus, $\mathcal F_{\alpha_i,j}=1$ if and only if $\binom{j}{\alpha_i}\equiv 0 \pmod{2}$. Therefore, we obtain (\ref{Y}) as required.
\end{proof}

The above proposition provides an explicit description of the forbidden sets $Y_{\alpha_i}$ in terms of binomial coefficients modulo $2$. 
This characterization allows us to translate the poset admissibility condition into modular constraints on the binary index $b(v)$. 
Consequently, the problem of determining poset vectors can be reformulated as a sieve problem on the interval $[0,2^n-1]$, where at each level $i$ one excludes the residue classes belonging to $Y_{\alpha_i}$.

\begin{theorem}[Sieve Theorem]\label{main}
Let $A\in\mathcal{PM}(n)$ and let $\alpha(A)=\{\alpha_0,\ldots,\alpha_{n-1}\}$ be its Pascal index set such that $A={\bf P}_{2^n}[\alpha]$. 
Let $v\in\mathbb{B}^n$ with binary index $b(v)\in[0,2^n-1]$. Then $v$ is a poset vector of $A$ if and only if for each $i=1,\ldots,n-1$,
$$
b(v)\bmod 2^{i+1}\notin Y_{\alpha_i}
$$
Equivalently, the set of binary indices of all admissible poset vectors of $A$ is given by
$$
\left\{\,x\in[0,2^n-1]\;\middle|\;
x\bmod 2^{i+1}\notin Y_{\alpha_i}
\text{ for some}\;\; i=1,\ldots,n-1
\right\},
$$
which is obtained by removing from $[0,2^n-1]$ all integers $x$ such that $x\bmod 2^{i+1}\in Y_{\alpha_i}$ for each $i=1,\ldots,n-1$.
\end{theorem}
\begin{proof} Let $A\in\mathcal{PM}(n)$ such that $A={\bf P}_{2^n}[\alpha]$ where $\alpha(A)=\{\alpha_0,\ldots,\alpha_{n-1}\}$.
Now let $v=(v_0,\ldots,v_{n-1})\in\mathbb{B}^n$ and let $x=b(v)\in[0,2^n-1]$ be its binary index. That is, $v$ identifies with $x$ via its binary expansion
$$
x=\sum_{j=0}^{n-1} v_j2^j.
$$
By Theorem \ref{ppp}, $v$ is a poset vector of $A$ if and only if $\mathrm{supp}(v)$ forms an order ideal in the associated poset $P_A$. 
Since $A={\bf P}_{2^n}[\alpha]$ with $\alpha=\{\alpha_0,\ldots,\alpha_{n-1}\}$, the order relations in $P_A$ are encoded by 
$$
i \preceq j \quad \Longleftrightarrow \quad \binom{\alpha_j}{\alpha_i}\equiv 1 \pmod{2}.
$$
Hence, the compatibility between $x$ and the existing elements $\alpha_i$ is determined by $\binom{x}{\alpha_i}\pmod{2}$. It implies that
$$
v_i=1\quad \Longleftrightarrow \quad \binom{x}{\alpha_i}\equiv 1\pmod{2}.
$$
Since $\alpha_i <2^{i+1}$ it follows from Lucas theorem \cite{ELU} that
$$
\binom{x}{\alpha_i}\equiv \binom{x\bmod 2^{i+1}}{\alpha_i}\pmod{2}.
$$
Thus, the compatibility condition at level $i$ depends only on $x\bmod 2^{i+1}$. By Proposition \ref{YY}, the set $Y_{\alpha_i}$ consists precisely of those residues that violate the admissibility condition at level $i$. 
Hence, a violation occurs at level $i$ if and only if
$$
x\bmod 2^{i+1}\in Y_{\alpha_i}.
$$
Therefore, $v$ is a poset vector of $A$ if and only if no violation occurs at any level, that is,
$$
x\bmod 2^{i+1}\notin Y_{\alpha_i}
\quad \text{for all } i=1,\ldots,n-1.
$$
Equivalently, starting from the full set $[0,2^n-1]$, we remove all integers $x$ that violate at least one level condition, namely those satisfying
$$
x\bmod 2^{i+1}\in Y_{\alpha_i}
\quad \text{for some } i=1,\ldots,n-1.
$$
The remaining integers are those that satisfy all level-wise constraints, and hence correspond bijectively to the admissible poset vectors of $A$.
This completes the proof.
\end{proof}

\begin{example}{\rm
Let 
$$
A=\begin{bmatrix}
1&0&0&0\\
1&1&0&0\\
1&0&1&0\\
1&1&1&1
\end{bmatrix}\in\mathcal{PM}(4).
$$
Then $\alpha(A)=\{1,3,5,15\}$. By (\ref{Y}), we obtain $Y_{\alpha_1=3}=\{2\}$, $Y_{\alpha_2=5}=\{4,6\}$, $Y_{\alpha_3=15}=\{8,9,\ldots,14\}$.
Let $x=b(v)\in [0,15]$. We remove from $[0,15]$ all integers satisfying
$$
x \equiv 2 \pmod{2^2},\quad
x \equiv 4,6 \pmod{2^3},\quad
x \equiv 8,\ldots,14 \pmod{2^4}.
$$
Applying the sieve step by step, we obtain
$$
\begin{aligned}
[0,15]
&\longrightarrow \{0,1,{\cancel 2},3,4,5,{\cancel 6},7,8,9,{\cancel 10},11,12,13,{\cancel 14},15\}\\
&\longrightarrow \{0,1,3,{\cancel 4},5,7,8,9,11,{\cancel 12},13,15\}\\
&\longrightarrow \{0,1,3,5,7,{\cancel 8},{\cancel 9},{\cancel 11},{\cancel 13},15\}.
\end{aligned}
$$
Hence the admissible binary indices are $\{0,1,3,5,7,15\}$. These correspond to the poset vectors in ${\mathbb B}^4$:
$$
(0,0,0,0),\ (1,0,0,0),\ (1,1,0,0),\ (1,0,1,0),\ (1,1,1,0),\ (1,1,1,1).
$$
}
\end{example}

\section{Enumeration through the automorphism groups}

In this section, we study how the action of the automorphism group ${\rm Aut}(A)$ on the poset vector space $\mathsf{PV}_n(A)$ determines the structure of $v$-extensions and leads to a recursive enumeration of $\mathcal{NL}(n+1)$.

Let $\mathfrak{S}_n$ denote the set of all permutations on $X_n$, and $\mathfrak{P}_n$ the set of all $n \times n$ permutation matrices.
Since an automorphism of a poset $P=(X_n,\preceq)$ is a bijection $X_n\to X_n$ that preserves the partial order,
the automorphism group of $P\in\mathcal{NL}(n)$ is defined by
$$
\operatorname{Aut}(P) := \{\sigma \in \mathfrak{S}_n \mid i \preceq j\; \text{ if and only if }\;\sigma(i) \preceq \sigma(j)\}.
$$
Accordingly, the automorphism group of a poset matrix $A\in {\cal PM}(n)$ is defined by
$$
\operatorname{Aut}(A) := \{Q\in \mathfrak{P}_n \mid Q^T A Q = A\}.
$$
Clearly, $\mathrm{Aut}(P)\cong{\rm Aut}(A_P)$.
\medskip

Let $G:=\mathrm{Aut}(A)$. Consider the automorphism group $G$ acting on $\textsf{PV}_n(A)$. Two poset vectors $v,w\in\textsf{PV}_n(A)$ are called {\it $G$-equivalent} if $vQ=w$ for some $Q\in G$. The relation of $G$-equivalence is an equivalence relation, and its equivalence class is called the {\it $G$-orbit}. In particular, the $G$-orbit of $v\in \textsf{PV}_n(A)$ is
$$
{\cal O}_A(v):=v\cdot G=\big\{vQ\mid Q\in G\big\}.
$$ 
By Lagrange's theorem, 
$$
|{\cal O}_A(v)|=|G|/|G_v|
$$
where $G_v=\{Q\in G\mid vQ=v\}$ is the stabilizer subgroup of $G$.

\begin{theorem}\label{class}
Let $A\in\mathcal{PM}(n)$ and let $v,w\in \textsf{PV}_n(A)$. Then $A^v\cong A^w$ if and only if $v$ and $w$ are $G$-equivalent, equivalently $w\in {\cal O}_A(v)$. \end{theorem}
\begin{proof} Let $w\in {\cal O}_A(v)$. Then there exists a $Q\in G$ such that $w=vQ$. Let $\hat Q$ be the permutation matrix in $\mathfrak{P}_{n+1}$ of the form  $\hat Q=\begin{bmatrix} Q & 0\\ 0 & 1\end{bmatrix}$. Then
\begin{eqnarray*}\label{QAQ1}
\hat Q^T\, A^v\, \hat Q =
\begin{bmatrix}
Q^TAQ & 0\\ vQ & 1
\end{bmatrix}
=
\begin{bmatrix}
A & 0\\ vQ & 1
\end{bmatrix}
= A^{vQ}=A^w,
\end{eqnarray*}
which shows $A^v\cong A^{w}$. 

Conversely, assume $A^v\cong A^{w}$. Then there exists a permutation matrix 
$\hat Q$ of order $n+1$ such that $\hat Q^T A^v\hat Q = A^{w}$. In any $v$-extension, the last column is the unique column whose entries above the diagonal are all zero. Thus $\hat Q$ must fix the last index, so we may assume for some $Q\in \mathfrak{P}_n$,
$$
\hat Q=\begin{bmatrix} Q & 0\\ 0 & 1\end{bmatrix}.
$$
It follows from $\hat Q^T A^v \hat Q = A^{w}$ that $Q^T A Q=A$ and $w=vQ$. Hence $w\in {\cal O}_A(v)$, which completes the proof.
\end{proof}

 The set of all $G$-orbits is denoted by
$$
\textsf{PV}_n(A)/G:=\big\{ {\cal O}_A(v)\mid v\in\textsf{PV}_n(A)\big\}.
$$ 

\begin{corollary}\label{PVn} Let $A\in\mathcal{PM}(n)$. Then $|\textsf{PV}_n(A)/G|$ is equal to the number of posets in ${\cal NIP}(n+1)$ obtained from $v$-extensions of $A$.
\end{corollary}

Let $\mathsf{NIP}_{n+1}(A)\subseteq {\cal NIP}(n+1)$ be the set of all non-isomorphic posets obtained from $v$–extensions of $A\in{\cal PM}(n)$.
Then it follows from Corollary \ref{PVn} that 
\begin{eqnarray*}\label{Birknumber}
|{\mathcal{NIP}}(n+1)|=\sum_{A\in{\cal PM}(n)} |{\mathsf{NIP}}_{n+1}(A)|=\sum_{A\in{\cal PM}(n)}|\textsf{PV}_n(A)/G|.
\end{eqnarray*}
This shows that an answer of the Birkhoff's question can be recursively obtained from $|\textsf{PV}_n(A)/G|$. By Burnside’s counting theorem, 
\begin{equation*}\label{burnside}
|\textsf{PV}_n(A)/G|\;=\;{1\over |G|}\sum_{Q\in G}{\rm{Fix}}(Q).
\end{equation*}
where ${\rm{Fix}}(Q)$ is the number of poset vectors $v$ in $\textsf{PV}_n(A)$ fixed by $Q\in G$, i.e., $vQ=v$.
\medskip

If $G$ is the identity group $\{{\rm id}\}$, then $G$ is said to be {\it trivial}.  
By Theorem \ref{class}, all $v$-extensions of $A$ are pairwise non-isomorphic when $A$ has trivial
automorphism group. Thus we have the following corollary.
\begin{corollary}
Let $A\in\mathcal{PM}(n)$. If $Aut(A)$ is trivial, then 
$$
|\mathsf{NIP}_{n+1}(A)|=|\mathsf{NL}_{n+1}(A)|.
$$
\end{corollary}

Understanding when a poset matrix $A$ has trivial automorphism group is of fundamental importance. Since $\mathrm{Aut}(A) \cong \mathrm{Aut}(P_A)$, it follows from Theorem \ref{class} that $\mathrm{Aut}(A)$ is trivial if and only if
the only order-preserving bijection on $P_A$ is the identity permutation.

The following theorems give sufficient conditions for this property.

\begin{theorem}\label{supp} Let $A\in\mathcal{PM}(n)$. If all poset vectors in $\textsf{PV}_n(A)$ have pairwise distinct support sizes, then $\mathrm{Aut}(A)$ is trivial.
\end{theorem} 
\begin{proof} Let $v,w\in \textsf{PV}_n(A)$. By Theorem \ref{class}, $A^v\cong A^w$ if and only if $Q^T A Q=A$ and $w=vQ$ for some  $Q\in \mathfrak{P}_n$. Thus, if $|{\rm supp}(v)|\neq |{\rm supp}(w)|$ for all pairs $(v,w)$ with $v\ne w$ in $\textsf{PV}_n(A)$ then there is no permutation $Q\in {\rm Aut}(A)$ such that $w=vQ$. 
\end{proof}

\begin{theorem}\label{thm36} Let $A\in\mathcal{PM}(n)$ be a poset matrix with $i$th row sum $r_i(A)$ and $i$th column sum $c_i(A)$.
If the pairs $\bigl(r_i(A),\,c_i(A)\bigr)$ are distinct for all $i\in[n]$, then $\mathrm{Aut}(A)$ is trivial.
\end{theorem}

\begin{proof}
Let $Q\in\mathrm{Aut}(A)$. Then the map $A\mapsto Q^TAQ$ yields 
 $A_{ij}=(Q^TAQ)_{\sigma(i),\,\sigma(j)}$ for all $i,j\in[n]$. Thus  
$$
r_i(A)=r_{\sigma(i)}(Q^TAQ)\quad {\text{and}}\quad c_i(A)=c_{\sigma(i)}(Q^TAQ).
$$
Since $Q^TAQ=A$ it follows that for all $i\in X_n$,
$$
\bigl(r_i(A),\,c_i(A)\bigr)\;=\;\bigl(r_{\sigma(i)}(A),\,c_{\sigma(i)}(A)\bigr).
$$
By hypothesis, the pairs $\bigl(r_i(A),c_i(A)\bigr)$ are distinct for all $i\in[n]$, so necessarily
$\sigma(i)=i$ for all $i$. Thus $Q=I_n$, and $\mathrm{Aut}(A)$ is trivial, which completes the proof.
\end{proof}
\medskip

Recall that $\downarrow i=\{j\in P\mid j \preceq i\}$ is the down set of $i\in P$. Accordingly, $\uparrow i =\{j\in P\mid i \preceq j\}$ is referred to as the {\it up set} of $i\in P$.
Since $A$ is the poset matrix of the NL poset $P_A$, it follows that $r_i(A)=|\downarrow i|$ and $c_i(A)=|\uparrow i|$. Thus we have the following corollary.
\begin{corollary} Let $P_A$ be the NL poset associated with the poset matrix $A\in\mathcal{PM}(n)$. If the pairs $\bigl(|\downarrow i|,|\uparrow i|\bigr)$ are distinct for all $i\in[n]$, 
then $\mathrm{Aut}(A)$ is trivial.
\end{corollary}

The {\it disjoint union} or {\it disjoint sum} of two posets $P$, $Q$, denoted by $P\bigsqcup Q$ or $P+Q$, is the 
poset $(P\bigcup Q,\preceq_+)$ where for all $x,y\in P\bigcup Q$, $x\preceq_+ y$ if and only if either $x\preceq_P y$ or $x\preceq_Q y$. A poset
 which cannot be written as the disjoint union of two posets is called {\it connected}. The {\it height} of $P$, denoted by $h(P)$, is the cardinality of a longest chain in $P$. Similarly, the {\it width} of $P$, denoted by $w(P)$, is the cardinality of a largest antichain in $P$. We denote ${\bf C}_n$ and ${\bf I}_n$ by the chain and antichain on $X_n$ respectively. 

\begin{theorem} {\rm(Dilworth \cite{Dilworth})} Let $P$ be a finite poset of width $w(P)=k$. Then
 there exists a partition of $P$ into $k$ chains, that is, $P={\bf C}_1\bigsqcup \cdots\bigsqcup {\bf C}_k$ 
 where ${\bf C}_i\bigcap {\bf C}_j=\emptyset$ for all $i\ne j$.
 \end{theorem}
  
In the following proposition, we give the automorphism groups of chains and antichains. The results immediately follow from Theorem \ref{thm36}.

\begin{proposition}  Let ${\bf C}_n$ and ${\bf I}_n$ be the chain and antichain of length $n$, respectively. Then
\begin{itemize}
\item[{\rm (a)}] $\mathrm{Aut}({\bf C}_n)\cong\ \{{\rm id}\}$;
\item[{\rm (b)}] $\mathrm{Aut}({\bf I}_n)\ \cong \ \mathfrak{S}_n$;
\item[{\rm (c)}] $\mathrm{Aut}({\bf I}_n\bigsqcup{\bf C}_m)\ \cong\ \mathfrak{S}_n$, $m\ge2$;
\item[{\rm (d)}] $\mathrm{Aut}({\bf C}_n\bigsqcup{\bf C}_m)\;\cong\;
\begin{cases}
\mathfrak{S}_2, & \text{if } n=m,\\[2mm]
 \{{\rm id}\}, & \text{if } n\neq m.
\end{cases}$
\end{itemize}
\end{proposition}

\begin{corollary} We have

\begin{itemize}
\item[{\rm (a)}] $|{\mathsf{NL}}_{n+1}({\bf C}_n)|=|{\mathsf{NIP}}_{n+1}({\bf C}_n)|=n+1$;
\item[{\rm (b)}] $|\mathsf{NL}_{n+1}({\bf I}_n)|=2^n$;\quad $|\mathsf{NIP}_{n+1}({\bf I}_n)|=n+1$.
\end{itemize}
\end{corollary}

The {\it ordinal sum} of two posets $P$ and $Q$, denoted by $P \oplus Q$, is the poset $(P\bigcup Q,\preceq_{\oplus})$ defined as follows: for all $x,y \in P\bigcup Q$, $x \preceq_{\oplus} y$ if and only if either $x \preceq_P y$; $x \preceq_Q y$; or $x \in P$ and $y \in Q$. Let $A$ and $B$ be the poset matrices of $P\in\mathsf{NL}_n$ and 
$Q\in\mathsf{NL}_m$, respectively. Then the poset matrix of the ordinal sum $P\oplus Q$ is given by
$$
A\oplus B
:=
\left[
\begin{array}{cc}
A & O\\
J & B
\end{array}
\right],
$$
where $O$ denotes the $n\times m$ zero matrix and $J$ denotes the 
$m\times n$ all-ones matrix (see \cite{Cheon1}).

\begin{theorem}
Let $A\in\mathcal{PM}(n)$ and $B\in\mathcal{PM}(m)$. Then
$$
{\rm Aut}(A\oplus B)\cong {\rm Aut}(A)\times {\rm Aut}(B).
$$
\end{theorem}
\begin{proof} Let $M:=A\oplus B$, and suppose $Q\in {\rm Aut}(M)$ is a permutation matrix of order $n+m$. Write 
$$
M=\begin{bmatrix}
A & O\\
J & B
\end{bmatrix}\in\mathcal{PM}(n+m)\quad{\text{and}}\quad
Q=\begin{bmatrix}
Q_1 & X\\
Y & Q_2
\end{bmatrix}\in \mathfrak{P}_{n+m},
$$
where $Q_1$ and $Q_2$ are (0,1)-matrices of orders $n$ and $m$ respectively.
Since $Q\in {\rm Aut}(M)$, we have $QMQ^{T}=M$. We first show that $X=O$ and $Y=O$. Let $\sigma$ be the permutation of the set $[n+m]$ corresponding to the permutation matrix $Q$. If $\sigma(i)\le n$ for all $i\le n$ then the bijectivity of $\sigma$ implies that $X=O$ and $Y=O$. Assume that $\sigma(i)>n$ for some $i\le n$. By the bijectivity of $\sigma$, there exists $j>n$ such that $\sigma(j)<n$. Thus, the entry $(i,j)$ lies in the upper-right block $O$ of $M$ and  $(\sigma(i),\sigma(j))$ lies in the lower-left block $J$ of $M$. Since $QMQ^{T}=M$, we have $M_{ij}=M_{\sigma(i),\sigma(j)}$ for all $i,j\in[n+m]$. It follows that $0=M_{ij}=M_{\sigma(i),\sigma(j)}=1$, a contradiction. Hence $\sigma(j)>n$ for every $j>n$ so that $X=O$ and $Y=O$ by the bijectivity of $\sigma$. 
Hence
$$
Q=
\begin{bmatrix}
Q_1 & O\\
O & Q_2
\end{bmatrix}
$$
where $Q_1$ and $Q_2$ are permutation matrices of orders $n$ and $m$ respectively.

Computing under Boolean operation yields
$$
QMQ^T
=\begin{bmatrix}
Q_1 & O\\
O & Q_2
\end{bmatrix}
\begin{bmatrix}
A & O\\
J & B
\end{bmatrix}
\begin{bmatrix}
Q_1^T & O\\
O & Q_2^T
\end{bmatrix}
=
\begin{bmatrix}
Q_1AQ_1^T & O\\
J & Q_2BQ_2^T
\end{bmatrix}.
$$
Thus $QMQ^T=M$ is equivalent to $Q_1AQ_1^T=A$ and $Q_2BQ_2^T=B$. That is, $Q_1\in {\rm Aut}(A)$ and $Q_2\in {\rm Aut}(B)$.
Hence ${\rm Aut}(A\oplus B)\cong{\rm Aut}(A)\times{\rm Aut}(B)$, as required.
\end{proof}

\begin{theorem}\label{congruence}
Let $A\in\mathcal{PM}(n)$ and let $B=Q^TAQ$ for any $Q\in\mathfrak{P}_n$. Then ${\rm Aut}(A)\cong {\rm Aut}(B).$
\end{theorem}
\begin{proof} Since 
$${\rm Aut}(A)=\{Q\in \mathfrak{P}_n |Q^TAQ=A\}\cong\{\tau\in \mathfrak{S}_n |Q_\tau^TAQ_\tau=A\},$$  we may define a map $\varphi:{\rm Aut}(A)\to {\rm Aut}(B)$ by $\varphi(\tau)=\sigma^{-1}\tau\sigma$ for any $\sigma\in \mathfrak{S}_n$. We claim that $\varphi$ is a group isomorphism. Using $B=Q^T A Q$ and $\tau\in{\rm Aut}(A)$, we obtain
\begin{eqnarray*}
Q_{\varphi(\tau)}^{T} B Q_{\varphi(\tau)}
&=& (Q_{\sigma}^{T} Q_{\tau}^{T} Q_{\sigma^{-1}}^{T})(Q_{\sigma}^{T} A Q_{\sigma})(Q_{\sigma^{-1}} Q_{\tau} Q_{\sigma})=Q_{\sigma}^T\big(Q_{\tau}^T A Q_{\tau}\big) Q_{\sigma}\\
&=&Q_{\sigma}^T A Q_{\sigma}=B.
\end{eqnarray*}
Thus $\varphi(\tau)\in {\rm Aut}(B)$. To show that $\varphi$ is surjective, let $\rho\in {\rm Aut}(B)$ and set $\tau:=\sigma\rho\sigma^{-1}$. Then it can be similarly shown that $\tau\in  {\rm Aut}(A)$. Also we have ${\rm{ker}}(\varphi)=\{1_\sigma\}$. Moreover, since 
\begin{eqnarray*}
\varphi(\tau\tau')=\sigma^{-1}\tau\tau'\sigma= (\sigma^{-1}\tau\sigma)(\sigma^{-1}\tau'\sigma)=\varphi(\tau)\varphi(\tau'),
\end{eqnarray*}
$\varphi$ is a homomorphism. Hence $\varphi$ is a group isomorphism, which completes the proof.
\end{proof}
We note that the matrix $B$ in Theorem \ref{congruence} is {\it not necessarily} lower triangular, however $B$ is a Boolean idempotent matrix \cite{kim}. This shows that automorphism groups of permutation similar matrices to a poset matrix are isomorphic up to relabeling.

\section{Computing automorphism groups via twin classes}
For a poset matrix $A \in \mathcal{PM}(n)$ with associated NL poset $P_A$, we have ${\rm Aut}(A)\cong {\rm Aut}(P_A)$. In general, computing ${\rm Aut}(A)$ directly from $A$ is not straightforward. To address this, we use a structural approach based on the notion of {\it twin elements} of a poset introduced in \cite{Sagan}. In this section we refine this concept by organizing such elements into {\it twin classes}, which provide a more transparent description of the internal symmetries of a poset. 

Following \cite{Sagan}, two elements $x,y \in P$ are called \emph{twins} if $\downarrow x =\; \downarrow y$ and $\uparrow x =\; \uparrow y$. Recall that down set $\downarrow x$ and up set $\uparrow x$ are defined using $\preceq$. This definition is often too restrictive for our purposes, since it typically yields only trivial equivalence classes. For instance, in an antichain, every element forms a singleton twin class under this definition. To better capture structural similarity, we instead use the strict order. Define the {\it strict} down and up sets of $x\in P$, respectively by 
$$
D(x)=\{z \in P \mid z \prec x\}, \qquad
U(x)=\{z \in P \mid x \prec z\}.
$$
We say that elements $x,y\in P$ are {\it twins}, written $x \equiv y$, if
$$
D(x)=D(y) \quad \text{and} \quad U(x)=U(y).
$$
It is straightforward to verify that $\equiv$ is an equivalence relation on $X_n$. We denote the set of equivalence classes by
$$
X_n/_{\equiv} = \{\alpha_1, \ldots, \alpha_r\},
$$
and refer to $\alpha_1, \ldots, \alpha_r$ as the {\it twin classes} of $P$. This definition reflects the comparability of $P$ and leads to a more meaningful partition of $X_n$ into twin classes.

Moreover, the twin classes admit a clear interpretation in terms of the associated poset matrix. In the poset matrix $A$, twin elements correspond precisely to identical rows and identical columns. Such repetitions reflect symmetries of the poset, since elements within the same twin class can be permuted freely without affecting the order relation. Consequently, twin classes provide a natural and effective tool for analyzing and computing the automorphism group.

The following proposition gives a matrix theoretic characterization of this phenomenon and serves as a key step in our approach. For $n\times n$ matrices $A$ and $B$, we say that $A$ is {\it permutation similar} to $B$ if there exists a permutation matrix $Q\in\mathfrak{P}_n$ such that $B=Q^TAQ$.

\begin{proposition}\label{block}
Let $A\in{\mathcal PM}(n)$ be a poset matrix, and suppose that $A$ is permutation similar to a matrix $B$ of the form
$$
B=
\begin{bmatrix}
I_k & X\\[2pt]
Y   & Z
\end{bmatrix},
$$
where $X,Y,Z$ are empty when $k=n$. For any permutation matrix $Q_k\in\mathfrak{P}_k$, define
$$
Q :=
\begin{bmatrix}
Q_k & 0\\
0 & I_{n-k}
\end{bmatrix}\in\mathfrak{P}_n.
$$
Then $Q\in {\rm Aut}(B)$ if and only if all rows of $X$ are identical and all columns of $Y$ are identical in the matrix $B$.
\end{proposition}
\begin{proof}
By definition, $Q\in {\rm Aut}(B)$ if and only if $Q^T B Q = B$. A direct block multiplication yields
$$
Q^{T} B Q
=
\begin{bmatrix}
I_k & Q_k^{T}X\\[2pt]
YQ_k & Z
\end{bmatrix}.
$$
Hence $Q^{T} B Q = B$ holds if and only if $Q_k^{T}X = X$ and $YQ_k = Y$. Since $Q_k$ is any permutation matrix in $\mathfrak{P}_k$, the condition
$Q_k^{T}X = X$ holds when every row of $X$ is identical,
while $YQ_k = Y$ holds when every column of $Y$ is identical.
This completes the proof.
\end{proof}
\begin{remark} {\rm Permutation similarity of poset matrices corresponds to relabeling the elements of the associated poset.
Such a relabeling preserves the isomorphism class of the poset, but it does not necessarily preserve the natural labeling.
Consequently, a matrix permutation similar to a poset matrix need not itself be a lower triangular matrix.} 
\end{remark}

Recall that for any subsets $\alpha,\;\beta\subseteq X_n$, we denote $A[\alpha\mid\beta]$ by the $|\alpha|\times |\beta|$ submatrix of $A$ obtained from $A$ by taking rows and columns indexed in $\alpha$ and $\beta$, respectively. If $\alpha=\beta$, we simplify the notation to $A[\alpha]$, and $\alpha^C=X_n\setminus \alpha$.
\begin{theorem}\label{classes} Let $A\in\mathcal{PM}(n)$ and let $\{\alpha_1,\ldots,\alpha_r\}$ be the set of twin classes of $P_A\in\mathcal{NL}(n)$.
Then for any $\alpha_i$ with $|\alpha_i|=k_i$,
\begin{enumerate}
\item[\rm(a)] $A[\alpha_i]=I_{k_i}$, equivalently $\alpha_i$ is the antichain of $k_i$ elements;
\item[\rm(b)] all rows of $A[\alpha_i\mid\alpha_i^C]$ are identical, equivalently if $x,y\in\alpha_i$ then $D(x)=D(y)$;
\item[\rm(c)] all columns of $A[\alpha_i^C\mid\alpha_i]$ are identical, equivalently if $x,y\in\alpha_i$ then $U(x)=U(y)$.
\end{enumerate}
\end{theorem}
\begin{proof} Clearly, the results (a)-(c) holds for $\alpha_i$ with $|\alpha_i|=1$ or $n$. Let $2\le |\alpha_i|<n$ and suppose that $x,y\in\alpha_i$. If $x\prec y$ then $x\in D(y)=D(x)$ so that $x\in D(x)$,
 a contradiction. Thus no two elements in $\alpha_i$ are comparable, so it is an antichain. Hence $A[\alpha_i]=I_{k_i}$, which proves (a).

Moreover, since $x \equiv y$, we have $D(x)=D(y)$ and $U(x)=U(y)$. In matrix terms of $A$, by (a), $A_{xy}=A_{yx}=0$, and clearly $A_{xx}=A_{yy}=1$. Thus, for every $z\in X_n\setminus\{x,y\}$ we have
$A_{xz}=A_{yz}$ and $A_{zx}=A_{zy}$, which proves (b) and (c) respectively.
\end{proof}

\begin{theorem}\label{welldef}
Let $P=(X_n,\preceq)\in\mathcal{NL}(n)$, and let $X_n/_{\equiv}=\{\alpha_1,\dots,\alpha_r\}$ be the set of twin classes of $P$.  
Define a relation $\preceq_{\equiv}$ on $X_n/_{\equiv}$ by $\alpha_i \preceq_{\equiv} \alpha_j$  if and only if $x \preceq y$ in $P$ for some $x\in\alpha_i,\ y\in\alpha_j$.
Then $(X_n/_{\equiv},\preceq_{\equiv})$ is a poset.
\end{theorem}
\begin{proof} First we show that $\preceq_{\equiv}$ is well defined. Let $\alpha_i\preceq_{\equiv}\alpha_j$. By definition there exist $x\in\alpha_i$ and $y\in\alpha_j$ such that $x\preceq y$ in $P$.
Suppose that $x\equiv x'$ and $y\equiv y'$. If $x=y$ then $\alpha_i=\alpha_j$. Otherwise, $x\prec y$ implies $x\in D(y)=D(y')$, so $x\prec y'$, hence $y'\in U(x)=U(x')$ and $x'\prec y'$.  
Thus $\alpha_i\preceq_{\equiv}\alpha_j$ is independent of the choice of representatives.

Reflexivity is immediate, and transitivity follows from that of $\preceq$.  
If $\alpha_i\preceq_{\equiv}\alpha_j$ and $\alpha_j\preceq_{\equiv}\alpha_i$, then for some $x\in\alpha_i$, $y\in\alpha_j$, we have $x\preceq y$ and $y\preceq x$, hence $x=y$ and $\alpha_i=\alpha_j$.  
Therefore, $\preceq_{\equiv}$ is a partial order on $X_n/_{\equiv}$.
\end{proof}

We refer to $(X_n/_{\equiv},\preceq_{\equiv})$ as the {\it twin poset} of $P$, and denote it by $P_{\equiv}$.

\begin{corollary}\label{twinposet}
Let $P\in\mathcal{NL}(n)$ and let $X_n/_{\equiv}=\{\alpha_1,\dots,\alpha_r\}$ be the set of twin classes of $P$.
Then the twin poset $P_{\equiv}$ is isomorphic to the NL poset on $X_r=\{0,\ldots,r-1\}$ ordered by $i\preceq j$ if and only if
$\alpha_{i+1}\preceq_{\equiv}\alpha_{j+1}$ for $i,j\in\{0,\ldots,r-1\}$.
\end{corollary}
\begin{proof}
For $i=0,\ldots,r-1$, choose a representative $x_i\in\alpha_{i+1}$ with $x_i=\min \alpha_{i+1}$ and write
$\alpha_{i+1}=\bar x_i$. Thus, we may assume that
$$
X_n/_{\equiv}=\{\bar x_0,\ldots,\bar x_{r-1}\},
\qquad
x_0<\cdots<x_{r-1}.
$$
The twin poset $P_{\equiv}$ is a subposet obtained from $P$ by contracting each twin
class $\bar x_i$ to its representative $x_i$.
By construction, $\bar x_i \preceq_{\equiv} \bar x_j$ if and only if $i\preceq j.$ Hence $P_{\equiv}$ is isomorphic to the NL poset $(X_r,\preceq)$, as required.
\end{proof}

\begin{theorem}\label{home} Let $P\in\mathcal{NL}(n)$ be an NL poset with twin classes $X_n/_\equiv=\{\alpha_1,\dots,\alpha_r\}$, and let $P_{\equiv}$ be the associated twin poset on $X_n/_\equiv$.  
For $\sigma\in \mathrm{Aut}(P)$, define a map $\pi:\mathrm{Aut}(P)\to \mathrm{Aut}(P_{\equiv})$ by
\begin{eqnarray*}\label{pisigma}
\pi(\sigma)(\alpha_i)=\sigma(\alpha_i):=\{\sigma(x)\mid x\in\alpha_i\},\quad (i=1,\ldots,r).
\end{eqnarray*}
Then $\pi$ is a group homomorphism and
\begin{eqnarray}\label{ker}
\ker(\pi)\cong \prod_{i=1}^r \mathfrak S_{|\alpha_i|},
\end{eqnarray}
where each symmetric group $\mathfrak S_{|\alpha_i|}$ acts on the twin class $\alpha_i$.
 \end{theorem}
\begin{proof} Let $P\in\mathcal{NL}(n)$. For $\sigma\in \mathrm{Aut}(P)$, define $\pi(\sigma)$ on $X_n/_\equiv$ by $\pi(\sigma)(\alpha_i)=\sigma(\alpha_i)$. 
First, we show that $\pi$ is well defined. Since $\sigma$ is an order-preserving
bijection on $X_n$, it preserves the twin relation: if $x\equiv y$, then $\sigma(x)\equiv \sigma(y)$. Hence $\sigma(\alpha_i)$ is again a twin class of $P$. Moreover, since
$\sigma$ is bijective, we have $|\sigma(\alpha_i)|=|\alpha_i|$. Thus $\pi(\sigma)$ is a well-defined bijection on the set $X_n/_\equiv$. It remains to check that $\pi(\sigma)$ preserves the order of $P_\equiv$.
Suppose $\alpha_i\preceq_\equiv \alpha_j$ in $P_\equiv$. Then, by definition, $x\preceq y$ for any $x\in\alpha_i$ and $y\in\alpha_j$. Since $\sigma$ is order-preserving, $\sigma(x)\preceq \sigma(y)$. Since $\sigma(x)\in\sigma(\alpha_i)$ and $\sigma(y)\in\sigma(\alpha_j)$,
we obtain $\sigma(\alpha_i)\preceq_\equiv \sigma(\alpha_j).$ Equivalently, 
$$\pi(\sigma)(\alpha_i)\preceq_\equiv \pi(\sigma)(\alpha_j).$$
Thus $\pi(\sigma)$ is an order-preserving bijection on $P_\equiv$. Therefore $\pi(\sigma)\in \mathrm{Aut}(P_\equiv)$, and hence $\pi$ is well defined.
\medskip

It is straightforward to show that $\pi$ is a group homomorphism. Let $\sigma,\tau\in {\rm Aut}(P)$. Then for each $\alpha_i\in X_n/{\equiv}$,
$$
\pi(\sigma\tau)(\alpha_i)
=(\sigma\tau)(\alpha_i)
=\sigma(\tau(\alpha_i))
=\pi(\sigma)\bigl(\pi(\tau)(\alpha_i)\bigr).
$$
Thus $\pi(\sigma\tau)=\pi(\sigma)\circ\pi(\tau)$, so $\pi$ is a group homomorphism.
\medskip

By definition, we have
$$\ker(\pi)=\{\sigma\in \mathrm{Aut}(P)\mid \pi(\sigma)(\alpha_i)=\sigma(\alpha_i)=\alpha_i\; \text{ for all } i=1,\ldots,r\}.$$ 
Thus every $\sigma\in \ker(\pi)$ permutes elements only within each twin class $\alpha_i$.
For each $i$, let $\mathfrak S_{|\alpha_i|}$ denote the symmetric group on $\alpha_i$.
Define a map
$$
\Phi:\ker(\pi)\longrightarrow \prod_{i=1}^r \mathfrak S_{|\alpha_i|}
$$
by $\Phi(\sigma) = (\sigma_1,\dots,\sigma_r)$, where $\sigma_i:=\sigma|_{\alpha_i}$ denotes the restriction of $\sigma$ to $\alpha_i$. 
Since every $\sigma\in \ker(\pi)$ satisfies $\sigma(\alpha_i)=\alpha_i$, each restriction $\sigma_i:\alpha_i\to\alpha_i$ is a bijection, hence $\sigma_i\in \mathfrak S_{|\alpha_i|}$.
Thus $\Phi$ is well-defined. To show that $\Phi$ is a group homomorphism, let $\sigma,\tau\in \ker(\pi)$. Since each $\alpha_i$ is $\sigma$- and $\tau$-invariant, we have $(\sigma\tau)_i= \sigma_i\circ \tau_i$. 
Hence 
$$
\Phi(\sigma\tau)=\bigl((\sigma\tau)_1,\ldots,(\sigma\tau)_r\bigr)=\bigl(\sigma_1\circ \tau_1,\ldots,\sigma_r\circ \tau_r\bigr)=\Phi(\sigma)\circ\Phi(\tau),
$$
as required. To prove injectivity, suppose $\Phi(\sigma)=\Phi(\tau)$. Then $\sigma_i=\tau_i$ for all $i$.
For any $x\in X_n$, there exists $i$ such that $x\in\alpha_i$, and hence
$$
\sigma(x)=\sigma_i(x)=\tau_i(x)=\tau(x).
$$
Thus $\sigma=\tau$, and $\Phi$ is injective.

To prove surjectivity, fix an ordering, $\alpha_i=\{x_{i1},\dots,x_{i|\alpha_i|}\}$. Let $\tau=(\tau_1,\dots,\tau_r)\in \prod_{i=1}^r \mathfrak S_{|\alpha_i|}$, and define $\sigma:X_n\to X_n$ by $\sigma(x_{ik}) := x_{i,\tau_i(k)}$. Since each $\tau_i$ is a permutation of $\{1,\dots,|\alpha_i|\}$,
$\sigma$ restricts to a bijection on each $\alpha_i$,
and hence $\sigma$ is a bijection on $X_n$. We show that $\sigma\in{\rm Aut}(P)$.
Let $x\in\alpha_i$ and $y\in\alpha_j$. If $i=j$, then $\sigma$ only permutes elements within the same twin class.
Since elements in a twin class have identical order relations,
all order relations are preserved. Let $i\neq j$. Then the relation between $x$ and $y$ depends only on the twin classes
$\alpha_i$ and $\alpha_j$. Since $\sigma(x)\in\alpha_i$, and $\sigma(y)\in\alpha_j$, we obtain $x\preceq y$ if and only if $\sigma(x)\preceq \sigma(y)$. Thus $\sigma$ preserves the order relation, so
$\sigma\in{\rm Aut}(P)$. Moreover, since $\sigma(\alpha_i)=\alpha_i$ for all $i$, we have $\sigma\in\ker(\pi)\subseteq{\rm Aut}(P)$.

Finally, we show that $\Phi(\sigma)=\tau$. For each $i$ and $k$,
$$
\sigma|_{\alpha_i}(x_{ik})
= \sigma(x_{ik})
= x_{i,\tau_i(k)}
= \tau_i(x_{ik}).
$$
Hence $\sigma_i=\tau_i$ so that $\Phi(\sigma)=(\sigma_1,\dots,\sigma_r)=(\tau_1,\dots,\tau_r)=\tau$. Therefore $\Phi$ is a group isomorphism, and we conclude that
$$
\ker(\pi)\cong \prod_{i=1}^r \mathfrak S_{|\alpha_i|},
$$
which proves (\ref{ker}).
\end{proof}

\begin{corollary}\label{surj} Let $\pi:{\rm Aut}(P)\to{\rm Aut}(P_\equiv)$ be the homomorphism defined in Theorem~\ref{home}. Then
$\pi$ is surjective if and only if every $\psi\in{\rm Aut}(P_{\equiv})$ preserves twin-class sizes,
that is,
$$
|\psi(\alpha_i)|=|\alpha_i|\quad\text{for all } \alpha_i\in X_n/_\equiv.
$$
\end{corollary}
\begin{proof} Assume that $\pi$ is surjective. Let $\psi\in{\rm Aut}(P_\equiv)$. Since $\pi$ is surjective, there exists $\sigma\in{\rm Aut}(P)$ such that
$\pi(\sigma)=\psi$. Suppose $\psi(\alpha_i)=\alpha_j$. Then, by the definition of $\pi$,
\[
\sigma(\alpha_i)=\pi(\sigma)(\alpha_i)=\psi(\alpha_i)=\alpha_j.
\]
Since $\sigma$ is a bijection on $X_n$, the restriction $\sigma|_{\alpha_i}:\alpha_i\to \alpha_j$ is a bijection. Hence
$$
|\alpha_i|=|\alpha_j|=|\psi(\alpha_i)|.
$$
Thus every $\psi\in{\rm Aut}(P_\equiv)$ preserves twin-class sizes.

\medskip
Conversely, assume that every $\psi\in{\rm Aut}(P_\equiv)$ preserves twin-class sizes. We show that ${\rm Im}(\pi)={\rm Aut}(P_\equiv)$. Clearly, ${\rm Im}(\pi)\subseteq {\rm Aut}(P_\equiv)$. Let $\psi\in{\rm Aut}(P_\equiv)$. With a fixed ordering of each twin class, $\alpha_i=\{x_{i1},\dots,x_{i|\alpha_i|}\}$,
define a map $\sigma:X_n\to X_n$ by $\sigma(x_{ik})=x_{jk}$ whenever $\psi(\alpha_i)=\alpha_j$.
This is well defined and bijective, since $\psi$ is a bijection on the set of
twin classes such that $|\alpha_i|=|\psi(\alpha_i)|=|\alpha_j|$. We claim that $\sigma\in{\rm Aut}(P)$.
Let $x\in\alpha_i$ and $y\in\alpha_{j}$. Then
$$
x\preceq y
\Longleftrightarrow
\alpha_i\preceq_\equiv \alpha_{j}
\Longleftrightarrow
\psi(\alpha_i)\preceq_\equiv \psi(\alpha_{j})
\Longleftrightarrow
\sigma(x)\preceq \sigma(y).
$$
Hence $\sigma$ preserves the order so that $\sigma\in{\rm Aut}(P)$. Finally, if $\psi(\alpha_i)=\alpha_j$, then
$\sigma(\alpha_i)=\alpha_j=\psi(\alpha_i)$. Therefore $\pi(\sigma)(\alpha_i)=\psi(\alpha_i)$ for every $i$, and hence $\pi(\sigma)=\psi$.
Thus $\psi\in{\rm Im}(\pi)$, which implies ${\rm Aut}(P_\equiv)\subseteq {\rm Im}(\pi)$, so $\pi$ is surjective.
Therefore, $\pi$ is surjective if and only if every automorphism of
$P_\equiv$ preserves twin-class sizes.
\end{proof}
\medskip

The {\it splitting lemma} is a fundamental theorem that characterizes split short exact sequences in the category of groups. This lemma asserts that for a short exact sequence $1 \to N \to G\xrightarrow{\;\pi\;} H \to 1$ where $N$ is normal in $G$, the sequence {\it splits} if there exists a group homomorphism $s:H\to G$ such that $\pi\circ s={\rm id}_H$. In this case, $G$ is isomorphic to the semidirect product $N\rtimes H$. In the following theorem, we give a canonical decomposition of the automorphism group of a NL poset. 

\begin{theorem}\label{aut} Let $\pi:{\rm Aut}(P)\to{\rm Aut}(P_\equiv)$ be the homomorphism defined in Theorem~\ref{home}. Then 
\begin{eqnarray}\label{decom}
{\rm Aut}(P)
\cong
\Bigl(\prod_{i=1}^r \mathfrak S_{|\alpha_i|}\Bigr)
\rtimes {\rm Im}(\pi).
\end{eqnarray}
\end{theorem}
 \begin{proof} Consider a homomorphism $\pi:{\rm Aut}(P)\to {\rm Aut}(P_\equiv)$ defined in Theorem~\ref{home}, that is,
$$
\pi(\sigma)(\alpha_i)=\sigma(\alpha_i)
\qquad(\alpha_i\in X_n/\!\equiv,\;\; i=1,\ldots,r).
$$
Let $N={\rm ker}(\pi)$ and $H={\rm Im}(\pi)\subseteq {\rm Aut}(P_{\equiv})$.
Then we have a short exact sequence
\begin{eqnarray}\label{seq}
1 \longrightarrow N \xrightarrow{\;f\;} {\rm Aut}(P) \xrightarrow{\;\pi\;} H \longrightarrow 1,
\end{eqnarray}
where ${\rm Im}(f)=N={\rm ker}(\pi)$. To show that this sequence splits, we construct a homomorphism $s:H\to{\rm Aut}(P)$ such that
$\pi\circ s={\rm id}_H$. Let $\psi\in H={\rm Im}(\pi)$. Then there exists a $\sigma\in{\rm Aut}(P)$
such that $\pi(\sigma)=\psi$. In particular, $\pi(\sigma)(\alpha_i)=\sigma(\alpha_i)=\psi(\alpha_i)=\alpha_j$ for some $j$, and since $\sigma|_{\alpha_i}$ is a bijection on $\alpha_i$, we have $|\alpha_i|=|\alpha_j|$. Now let $\alpha_i=\{x_{i1},\dots,x_{i|\alpha_i|}\}$. Define $s(\psi)(x_{ik}) := x_{jk}$ whenever $\psi(\alpha_i)=\alpha_j$.
Since elements in the same twin class have identical order relations
with all other elements, mapping the element $x_{ik}\in\alpha_i$ to the element $x_{jk}\in\alpha_j$ preserves all order relations.
Hence $s(\psi)\in{\rm Aut}(P)$. 

Moreover, we have
$$
s(\psi)(\alpha_i)
= \{\,s(\psi)(x_{ik}) \mid x_{ik}\in\alpha_i\}
= \{\,x_{jk} \mid k=1,\dots,|\alpha_j|\,\}
= \alpha_j.
$$
It follows that $\pi(s(\psi))(\alpha_i)= s(\psi)(\alpha_i)= \alpha_j= \psi(\alpha_i)$.
Since this holds for all twin classes $\alpha_i$, we conclude that $\pi(s(\psi))=\psi$. Therefore, $\pi\circ s = {\rm id}_H$.
For $\psi_1,\psi_2\in H$, suppose $\psi_2(\alpha_i)=\alpha_j$ and $\psi_1(\alpha_j)=\alpha_\ell$. Then $s(\psi_1\psi_2)(x_{ik})=x_{\ell k}$, while
$$
(s(\psi_1)\circ s(\psi_2))(x_{ik})
= s(\psi_1)(x_{jk})
= x_{\ell k}.
$$
Thus $s(\psi_1\psi_2)=s(\psi_1)\circ s(\psi_2)$ so that $s$ is a group homomorphism. Therefore, the exact sequence (\ref{seq}) splits and hence by splitting lemma, we obtain (\ref{decom})
as required.\end{proof}
  
The decomposition (\ref{decom}) is closely related to the {\it wreath products} \cite{Okada} of the symmetric groups where local symmetries and global symmetries combine through semidirect products. 
Recall that a wreath product combines a group acting independently on several copies of a set together with another group permuting those copies. 
Similarly, in our setting, each symmetric group $\mathfrak S_{|\alpha_i|}$ acts on the twin class $\alpha_i$, while the symmetry of ${\rm Im}(\pi)$ acts on the collection of twin classes themselves. 
The resulting decomposition is useful for understanding orbit structures
and Burnside-type enumeration formulas for nonisomorphic posets.

\begin{corollary}
Let $\pi:{\rm Aut}(P)\to{\rm Aut}(P_\equiv)$ be the homomorphism defined in
Theorem~\ref{home}. If $\pi$ is surjective, then
\[
\bigl|{\rm Aut}(P)\bigr|
=
\left(\prod_{i=1}^r |\alpha_i|!\right)
\bigl|{\rm Aut}(P_\equiv)\bigr|.
\]
\end{corollary}

We end this section with an example.
\begin{example}\label{ex2}{\rm Consider the following poset $P\in\mathcal{NL}(7)$ and its poset matrix $A$: 

{\centering
\begin{minipage}{0.55\textwidth}
\centering
\begin{tikzpicture}[
  scale=0.7,
  every node/.style={circle, draw, inner sep=0.5pt, minimum size=16pt, font=\small},
  edge/.style={line width=0.8pt}]
  % Nodes
  \node (0) at (0,0.2) {0};
  \node (1) at (-1.2,1.5) {1};
  \node[fill=red!30] (2) at (0.1,1.5) {2};
  \node[fill=red!30] (3) at (1.2,1.5) {3};
  \node (4) at (-1.2,3) {4};
  \node (5) at (1.2,3) {5};
  \node (6) at (0,4.0) {6};

  % Edges (cover relations)
  \draw[edge] (0) -- (1);
  \draw[edge] (0) -- (2);
  \draw[edge] (0) -- (3);

  \draw[edge] (1) -- (4);
  
  \draw[edge] (2) -- (5);
  \draw[edge] (3) -- (5);

  \draw[edge] (4) -- (6);
  \draw[edge] (5) -- (6);

\end{tikzpicture}
\end{minipage}%
\begin{minipage}{0.4\textwidth}
\centering
$$
A=
\begin{bmatrix}
1&0&0&0&0&0&0\\
1&1&0&0&0&0&0\\
1&0&{\red 1}&{\red 0}&0&0&0\\
1&0&{\red 0}&{\red 1}&0&0&0\\
1&1&0&0&1&0&0\\
1&0&1&1&0&1&0\\
1&1&1&1&1&1&1
\end{bmatrix}.
$$
\end{minipage}}
\medskip

\noindent It is easy to see that 
\begin{itemize}
\item $X_7/_\equiv=\big\{\alpha_1=\{0\},\;\alpha_2=\{1\},\;\alpha_3=\{2,3\},\;\alpha_4=\{4\},\;\alpha_5=\{5\},\;\alpha_6=\{6\}\big\}$.
\item $A[\alpha_1]=A[\alpha_2]=A[\alpha_4]=A[\alpha_5]=A[\alpha_6]=I_1,\;A[\alpha_3]=I_2.$
\end{itemize}
By Theorem~\ref{home}, $\ker(\pi)\cong \mathfrak S_{|\alpha_3|}=\mathfrak S_2$. Indeed, 
$$
{\rm Aut}(P)=\{{\rm id},\;1243567\}\cong \mathfrak S_2,\quad{\text{and}}\quad \ker(\pi)=\{{\rm id},\;1243567\}={\rm Aut}(P).
$$
The associated twin poset $P_\equiv$ and its poset matrix are given in

{\centering
\begin{minipage}{0.55\textwidth}
\centering
\begin{tikzpicture}[
  scale=0.4,
  edge/.style={line width=0.9pt}
]
  \node (a0) at (0,0) {$\alpha_1$};
  \node (a1) at (-1.6,1.8) {$\alpha_2$};
  \node (a23) at (1.6,1.8) {$\alpha_3$};
  \node (a4) at (-1.6,3.6) {$\alpha_4$};
  \node (a5) at (1.6,3.6) {$\alpha_5$};
  \node (a6) at (0,5.2) {$\alpha_6$};

   \draw[edge] (a0) -- (a1);
   \draw[edge] (a0) -- (a23);
   \draw[edge] (a1) -- (a4);
   \draw[edge] (a23) -- (a5);
   \draw[edge] (a4) -- (a6);
   \draw[edge] (a5) -- (a6);
\end{tikzpicture}
\end{minipage}%
\begin{minipage}{0.4\textwidth}
\centering
$$
A_\equiv=
\begin{bmatrix}
1&0&0&0&0&0\\
1&1&0&0&0&0\\
1&0&1&0&0&0\\
1&1&0&1&0&0\\
1&0&1&0&1&0\\
1&1&1&1&1&1
\end{bmatrix}.
$$
\end{minipage}
}
\medskip

On the other hand, ${\rm Aut}(P_\equiv)=\{{\rm id},\;132546\}\cong \mathfrak S_2$, where the automorphism $\tau:=132546$ exchanges
$\alpha_2\leftrightarrow \alpha_3$ and $\alpha_4\leftrightarrow \alpha_5$. However, $|\alpha_2|\ne|\alpha_3|$, so $\tau$ does not preserve twin-class sizes.
Hence, by Corollary \ref{surj} $\pi$ is not surjective. Indeed, ${\rm Im}(\pi)=\{{\rm id}\}$ and by Theorem \ref{aut},
$$
{\rm Aut}(P)\cong\ker(\pi)\rtimes {\rm Im}(\pi)=\mathfrak S_2\rtimes\{{\rm id}\}\cong\mathfrak S_2.
$$
}
\end{example}

\section{Generating NL posets through topological growth of distributive Lattices}

In this section, we formalize a relationship between $v$-extensions of poset matrices of naturally labeled posets and naturally labeled topologies. The key bridge is Birkhoff's fundamental theorem for finite distributive lattices. 

By convention, each subset $S$ of $X_n$ is denoted by the weakly increasing word of its elements. For instance, $v=(1\;0\;1\;1\;0\;0\;0)\in{\mathbb B}^7$ is the characteristic vector of the
subset $\{0, 2, 3\}\subseteq X_7$ and it is denoted by $023$, and $\emptyset$ is denoted by the empty
word~$\epsilon$. For any lattice
$L$, an element $x$ is said to be {\it join-irreducible} if $x$ is not the join of a finite set of other elements. Let $\SetJoinIrreducibles(L)$ denote the set of join-irreducible elements of $L$, and recall that $\SetIdeals(P)$ is the set of order ideals of the poset $P$. This set $\SetIdeals(P)$ is a lattice ordered by set inclusion.

The fundamental theorem \cite{Birkhoff} of finite distributive lattices is stated as follows.

\begin{theorem} [Birkhoff's theorem]\label{Birk1} For any finite distributive lattice $L$, there is a unique poset $P$ such that $L$ is
isomorphic as a poset to the set $\SetIdeals(P)$ ordered by set inclusion.
\end{theorem}

By Birkhoff's theorem, every finite lattice $L$ is isomorphic to the {\it ideal lattice}, $\SetIdeals(P)$ for some poset $P$ (see Fig. 1). An important property related with this theorem is the following.
\begin{theorem}\cite{Birkhoff}
For any finite poset $P$, the subposet $\SetJoinIrreducibles(\SetIdeals(P))$ of
$\SetIdeals(P)$ is isomorphic to $P$.
\end{theorem}

For instance, by considering the poset $P\in\mathcal{NL}(7)$ in Example \ref{ex2}, we have:
\begin{itemize}
\item $L:=J(P)=\{\epsilon,0,01,02,03,014,012,013,023,0124,0134,0123,0235,01234,01235,012345,0123456\}$;
\item $J{\rm Irr}(L) = \{0,01,02,03,014,0235,0123456\}\cong P$.
\end{itemize}

\usetikzlibrary{positioning}
\tikzset{
  pn/.style={circle,inner sep=1.5pt,draw,fill=black},
  newn/.style={circle,inner sep=1.5pt,draw,fill=red!12},
  hasse/.style={line width=0.6pt},          
  redge/.style={line width=1pt,red},        
}

\begin{center}
\begin{tikzpicture}[
  scale=0.8,
  edge/.style={line width=0.8pt}]
  % Nodes
    \node (0hat) at (0,0) {$\epsilon$};
    
    \node (0) at (0,1) {{\red 0}};
    
    \node (1) at (-2,2) {{\red 01}};
    \node (2) at (0,2) {{\red 02}};
    \node (3) at (2,2) {{\red 03}};
    
    \node (12) at (-2,3) {012};
    \node (13) at (0,3) {013};
    \node (23) at (2,3) {023};
    \node (4) at (-4,3) {{\red 014}};

    \node (123) at (0,4) {0123};
    \node (34) at (-2,4) {0134};
    \node (24) at (-4,4) {0124};
    \node (5) at (2,4) {{\red 0235}};

    \node (15) at (0,5) {01235};
    \node (234) at (-2,5) {01234};

    \node (45) at (-2,6) {012345};
    
    \node (6) at (-2,7) {{\red 0123456}};
  
  % Edges (cover relations)
    \draw[edge] (0hat) -- (0);
    
    \draw[edge] (0) -- (1);
    \draw[edge] (0) -- (2);
    \draw[edge] (0) -- (3);

    \draw[edge] (1) -- (12);
    \draw[edge] (1) -- (13);
    \draw[edge] (2) -- (12);
    \draw[edge] (2) -- (23);
    \draw[edge] (3) -- (13);
    \draw[edge] (3) -- (23);
    \draw[edge] (1) -- (4);

    \draw[edge] (12) -- (123);
    \draw[edge] (12) -- (24);
    \draw[edge] (13) -- (123);
    \draw[edge] (13) -- (34);
    \draw[edge] (23) -- (123);
    \draw[edge] (23) -- (5);
    \draw[edge] (4) -- (24);
    \draw[edge] (4) -- (34);

    \draw[edge] (24) -- (234);
    \draw[edge] (34) -- (234);
    \draw[edge] (123) -- (234);
    \draw[edge] (123) -- (15);
    \draw[edge] (5) -- (15);

    \draw[edge] (15) -- (45);
    \draw[edge] (234) -- (45);

    \draw[edge] (45) -- (6);
\end{tikzpicture}

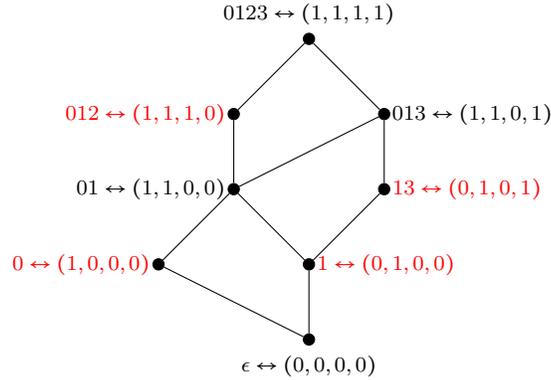
Fig. 1: Ideal lattice $L=\SetIdeals(P)$ and its sublattice $\SetJoinIrreducibles(L)\cong P$ (in red).  
\end{center}

\medskip

We now turn to the Birkhoff's theorem from the matrix point of view. We begin by obtaining a matrix characterization of poset vectors of the poset matrix $A$, which are exactly same as the characteristic vectors of order ideals of the poset $P_A$.

\begin{theorem}\label{fixedpoint}
Let $A\in\mathcal{PM}(n)$ be a poset matrix. Then $v\in {\mathbb B}^n$ is a poset vector of $A$ if and only if $vA = v$ under Boolean matrix product.
Equivalently,
$$
\SetIdeals(P_A)\;=\;\{\,\mathrm{supp}(v)\subseteq X_n\mid vA = v\,\}.
$$
\end{theorem}
\begin{proof}  Let $A=[a_{i,j}]\in {\cal PM}(n)$ and $v=(v_0,\ldots,v_{n-1})$ be a poset vector of $A$. Then
\begin{eqnarray}\label{fix}
(vA)_j=\sum_{i\in X_n} v_i a_{i,j}=1
\iff \exists i\in \mathrm{supp}(v)\ \text{with}\ j\preceq i.
\end{eqnarray}
By Theorem \ref{ppp}, $\mathrm{supp}(v)$ is an order ideal of $P_A$ so that $j\in \mathrm{supp}(v)$. Hence $(vA)_j\le v_j$ for all $j\in X_n$. In addition, since $A$ has ones on the diagonal, 
$$(vA)_j=\sum_{i}v_i a_{i,j}\ge v_j a_{j,j}=v_j,$$
so that $(vA)_j\ge v_j$ for all $j\in X_n$. Consequently, if $v$ is a poset vector of $A$ then $vA=v$.

Conversely, assume $vA=v$ for a row vector $v\in{\mathbb B}^n$. If $i\in\mathrm{supp}(v)$ and $j\preceq i$, then it follows from (\ref{fix}) that $(vA)_j=1$, so $j\in \mathrm{supp}(v)$. Therefore $\mathrm{supp}(v)$ is downward closed and it is an order ideal of $P_A$. By Theorem \ref{ppp}, $v$ is a poset vector of $A$.
\end{proof}

For a poset matrix $A\in\mathcal{PM}(n)$, define
$$
L_A \;:=\; \{\,x\in {\mathbb B}^n \mid xA=x \,\}.
$$

\begin{theorem}\label{Birk3} Let $A\in\mathcal{PM}(n)$ be a poset matrix. Then $L_A$ is a finite distributive lattice with the componentwise Boolean join $\vee$ and meet $\wedge$. Moreover, $L_A\cong \SetIdeals(P_A)$.
\end{theorem}

\begin{proof} Let $x\in L_A$. By Theorem~\ref{fixedpoint}, $x$ is a poset vector of $A$, Equivalently, by Proposition \ref{ppp}, $\mathrm{supp}(x)$ is an order ideal of $P_A$.
Thus the map $\Phi:L_A\to \SetIdeals(P_A)$ defined by $x\mapsto\mathrm{supp}(x)$ is a bijection. Moreover, for $x,y\in L_A$ 
it is easily shown that $\Phi(x\vee y)=\Phi(x)\cup \Phi(y)$ and $\Phi(x\wedge y)=\Phi(x)\cap \Phi(y)$, whence $\Phi$ is a lattice isomorphism.
Since $\SetIdeals(P_A)$ is a finite distributive lattice, so is $L_A$.
\end{proof}

 We call $L_A$ the {\it ideal lattice} of $P_A$. A poset vector $v\in {\mathbb B}^n$ of $A\in{\cal PM}(n)$ is said to be \emph{join-irreducible} if $v=u+w$ then $v=u$ or $v=w$, exclude $v={\bf 0}$. In the following proposition, we obtain a matrix characterization of join-irreducible poset vectors of $A$.

\begin{theorem}\label{Birk4} Let $A\in\mathcal{PM}(n)$ be a poset matrix with row vectors $r_0,\ldots,r_{n-1}$. Then $v$ is a join-irreducible poset vector of $A$ if and only if $v$ is a row vector of $A$.
\end{theorem}
\begin{proof} Let $v$ be a poset vector of $A\in{\cal PM}(n)$. By definition, $v$ is join-irreducible if and only if
${\rm{supp}}(v)$ is join-irreducible in the ideal lattice $\SetIdeals(P_A)$. Thus join-irreducible poset vectors correspond exactly to the characteristic vectors of principal order ideals of $P_A$, which are row vectors $r_i$ for $i\in X_n$ of $A$.
\end{proof}

By Theorems \ref{Birk1}-\ref{Birk4}, the matrix representation theorem for finite distributive lattices can be stated as follows.  

\begin{theorem}\label{MRB} {\rm (Matrix representation of Birkhoff's theorem)} For any finite distributive lattice $L$, there is a unique poset matrix $A\in{\cal PM}(n)$ such that 
$$ 
L \cong \{\,x\in {\mathbb B}^n\mid xA=x\,\}.
$$
Moreover, $\SetJoinIrreducibles(L)\cong \{r_i\in {\mathbb B}^n\mid i\in X_n\}$ where $r_i$ denotes the $i$th row vector of $A$.
\end{theorem}

For instance, let $L$ be the Boolean lattice $\mathbb{B}_3$. Since 
$$
L\cong\{x\in{\mathbb{B}^3}\mid xI_3=x\}=\{0,1\}^3,
$$
it follows that the poset matrix is $A=I_3$. Moreover,
$$
\SetJoinIrreducibles(L)\cong \{(1,0,0),\;(0,1,0),\;(0,0,1)\}.
$$ 

The matrix representation in Birkhoff’s theorem provides a natural bridge 
between algebraic, combinatorial, and topological aspects of distributive lattices.
For the poset matrix $A\in{\cal PM}(n)$ associated with the distributed lattice $L_A$, each 
row vector $r_i$ of $A$ corresponds to a join-irreducible element of $L_A$, while the 
Boolean fixed–point equation $xA=x$ encodes the closure of order ideals of the NL poset $P_A$
under union and intersection. 
\medskip

Recall that a topology of size $n$ is a set $\Gamma$ of subsets of $X_n$ such that $\emptyset\in \Gamma$, $X_n\in \Gamma$, and $\Gamma$ is closed for $\cup$ and $\cap$.
\begin{proposition} If $P$ is an NL poset of size $n$, then $\SetIdeals(P)$ is a topology of
size $n$.
\end{proposition}
\begin{proof} 
Let $P$ be an NL poset on the set $X_n$. Since $\SetIdeals(P)$ is the collection of order ideals of $P$, it immediately follows that $\SetIdeals(P)$ satisfies the axioms of a topology on $X_n$.  
\end{proof}
\medbreak

A topology $T$ is \emph{naturally labeled} if there exists an NL poset $P$ satisfying $T =
\SetIdeals(P)$. Let $\SetNLT$ be the graded set of naturally labeled topologies (abbreviated
as NLTs).

\begin{theorem} For any NLT $T$, there is a unique NL poset $P$ such that $T = \SetIdeals(P)$.
\end{theorem}
\begin{proof} Existence holds by the definition of NLT. It remains to prove uniqueness.
Suppose $P$ and $P'$ are NL posets on $X_n$ such that $\SetIdeals(P)=\SetIdeals(P')=:T$.
We show that the partial orders $\preceq_P$ and $\preceq_{P'}$ must coincide.
For $y\in X_n$ and an order ideal $I\in\SetIdeals(P)$, consider the subset of $X_n$ given by
$$
\downarrow_T y \;:=\; \bigcap\{\, I\in T \mid y\in I \,\}.
$$
Since $T=\SetIdeals(P)$ is closed under arbitrary intersections, the set $\downarrow_T y$ is the smallest order ideal containing $y$. Thus, $\downarrow_T y$ coincides with the principal ideal $\downarrow_P y$, i.e.,
$$
\downarrow_T y\;=\;\downarrow_P y \;=\; \{\, x\in X_n \mid x \preceq_P y \,\}.
$$ 
By the same argument with $P'$ in place of $P$, we also have $\downarrow_T y \;=\; \downarrow_{P'} y$.
Therefore, for all $y\in X_n$, $\downarrow_{P} y=\downarrow_{P'} y$ which proves $\preceq_P$ and $\preceq_{P'}$ must coincide. Thus the proof is completed.\end{proof}

%\section{Growing operation and combinatorial generation}
Let $T$ be an NLT of size $n$ and $S$ be an element of $T$. Let $\Grow_S(T)$ be the NLT of
size $n + 1$ obtained by {\it doubling the interval} $[S, X_n]$ of $T$ in such a way that the
doubled interval is obtained by inserting $n$ to each element of the copied interval. The interval doubling is a fundamental operation introduced in \cite{Day79}.
\medbreak

For instance, for $T = \{\epsilon, 0, 1, 01, 13, 012, 013,0123\}$, we have
\begin{eqnarray*}
    \Grow_{13}(T) &=& \{\epsilon, 0, 1, 01, 13, 012, 013, 0123, 134, 0134, 01234\};\\
 \Grow_{013}(T) &=& \{\epsilon, 0, 1, 01, 13, 012, 013, 0123, 0134, 01234\}.
\end{eqnarray*}

In the following theorem, we show that adjoining a new element with an element $S$ of $T$ corresponds to doubling the interval $[S,X_n]$ in the topology of ideals.

\begin{theorem} \label{thm:grow}
Let $A\in{\cal PM}(n)$ be the poset matrix of an NL poset $P$, and let
$T$ be the associated NLT of order ideals of $P$. If $v$ is the characteristic vector of any $S\in T$, then
\begin{eqnarray}\label{NLT}
\Grow_S(T)\;=\;\SetIdeals\!\bigl(P_{A^v}\bigr).
\end{eqnarray}
\end{theorem}
\begin{proof} For $S\in T=\SetIdeals(P)$, let $v$ be the characteristic vector of $S$. Then $v$ is the poset vector with ${\rm{supp}}(v)=S$. 
Thus, the $v$-extension $A^v$ is a poset matrix of size $n+1$ whose associated NL poset $P':=P_{A^v}$ contains the new element $x:=n$. Let $T'=\SetIdeals(P')$.
An ideal $I\in T'$ satisfies that either (i) does not contain $x$, hence $I\in \SetIdeals(P)$, or (ii) is of the form
$U\cup\{x\}$ where $U\in \SetIdeals(P)$ and $S\subseteq U$. Thus 
\begin{eqnarray*}
T'\;=\;T\ \bigcup\ \bigl\{\,U\cup\{x\}\ :\ U\in T,\ S\subseteq U\,\bigr\}.
\end{eqnarray*}
On the other hand, by the definition of doubling the interval $[S, X_n]$ of the distributive lattice
$T$, 
$$
\Grow_S(T)\;=\;T\ \bigcup\ \bigl\{\,U\cup\{n\}\ :\ U\in T,\ S\subseteq U\,\bigr\}.
$$
Identifying the new element $x$ in $P'$ with the label $n$, we obtain (\ref{NLT}), as required.
\end{proof}

\begin{theorem} \label{thm:increase_topology}
Let $P$ be a naturally labeled poset of size $n$. For any naturally
labeled poset $P'$ of size $n + 1$ such that the restriction of $P'$ on $X_n$ is $P$, there
is a unique $S \in \SetIdeals(P)$ such that $\Grow_S(\SetIdeals(P)) = \SetIdeals(P')$.
\end{theorem}
\begin{proof} From the hypotheses of the statement, it appears that the poset $P'$ is obtained from $P$ by
adding a new element $n$. Therefore, by Theorem~\ref{ppp}, there is a poset
vector $v$ of the poset matrix $A$ of $P$ such that the poset matrix $A'$ of $P'$ is the
$v$-extension of $A$. Let $S$ be the subset of $X_n$ such that $v$ is the characteristic
vector of $S$. By Theorem~\ref{thm:grow}, $\Grow_S(\SetIdeals(P)) = \SetIdeals(P')$.
The uniqueness of $S$ is a consequence of Theorem~\ref{ppp} and in particular of
the fact that $v$-extensions and order ideals of a NL poset are in one-to-one
correspondence.
\end{proof}

For instance, let the poset $(P, \Leq)$ on $X_5$ such that $0 \Leq 1 \Leq 2$, $1 \Leq 4$,
and $3 \Leq 4$. We have $\SetIdeals(P) = \{\epsilon, 0, 01, 3, 03, 012, 013, 0123, 0134,
01234\}$. By setting $P'$ as the poset obtained by adding to $P$ the element $5$ so that $2
\Leq 5$ and $3 \Leq 5$, the unique set $S \in \SetIdeals(P)$ such that
$\Grow_S(\SetIdeals(P)) = \SetIdeals(P')$ is $S = \{0, 1, 2, 3\}$.
\medbreak

The operation $\Grow_S$ on NLTs leads to the following obvious algorithm to recursively
generate all NLTs of size $n$.
\medbreak

\begin{enumerate}
    \item \textbf{function} GenerateNLTs($n$)
    \item \qquad \textbf{if} $n = 0$, \textbf{then}
    \item \qquad \qquad \textbf{return} $\{ \{\epsilon\} \}$
    \item \qquad \textbf{else}
    \item \qquad \qquad $R := \emptyset$
    \item \qquad \qquad \textbf{for each} $T \in$ GenerateNLTs($n - 1$)
    \item \qquad \qquad \qquad \textbf{for each} $S \in T$
    \item \qquad \qquad \qquad \qquad $R := R \cup \left\{\Grow_S(T)\right\}$
    \item \qquad \qquad \textbf{return} $R$
\end{enumerate}
\medbreak

This algorithm can be used to generate all NL posets of size $n$ by converting
an NLT $T$ into its corresponding NL poset obtained by considering the join-irreducible elements
of $T$. Such an approach is not adapted to compute the numbers of NLTs of a given size $n$
since GenerateNLTs needs to keep in memory all NLTs of sizes $n$ or smaller.
\medbreak

However, we can use the operation $\Grow_S$ in order to propose a first/next approach which
is very light in memory. See \cite{NW78} for a general description of first/next generation algorithms. For this, we need the following definitions.
\medbreak

Let, for any $n \in \N$,
\begin{equation*}
    \First(n) := \Grow_\epsilon^n(\{\epsilon\}).
\end{equation*}
For instance, $\First(3) = \Grow_\epsilon^3(\{\epsilon\}) = \{\epsilon, 0, 1, 2, 01, 02, 12,
012\}$.
\medbreak

For any NLT $T$ of size $n$, let $\Cut(T)$ be the NLT of size $n - 1$ obtained by deleting
all sets of $T$ which contain $n - 1$. For instance, $\Cut(\{\epsilon\}) = \{\epsilon\}$ and
\begin{equation*}
    \Cut(\{\epsilon, 0, 1, 01, 13, 012, 013, 0123\}) = \{\epsilon, 0, 1, 01, 012\}.
\end{equation*}
\medbreak

For any NLT $T$ of size $n$ and any $S \in T$, let $\Next_S(T)$ be the set of $T$ which is
immediately greater than $S$ for the length-lexicographic order on the word representation of the
subsets of $X_n$. For instance,
\begin{equation*}
    \Next_{13}(\{\epsilon, 0, 1, 01, 13, 012, 013, 0123\}) = 012.
\end{equation*}
Observe that this map $\Next_S$ is partial since, for instance, $\Next_{01}(\{\epsilon, 0,
1, 01\})$ is not defined.
\medbreak

Now, for any NLT $T$ of size $n$, let $\Next(T)$ be the NLT returned by the following
algorithm taking $T$ as input:
\begin{enumerate}
    \item \textbf{function} Next($T$)
    \item \qquad \textbf{let} $S$ be the smallest set of $T$ containing $n - 1$
    \item \qquad $S' := S \setminus \{n - 1\}$
    \item \qquad $T' := \Cut(T)$
    \item \qquad \textbf{if} $\Next_{S'}\left(T'\right)$ is defined, \textbf{then}
    \item \qquad \qquad $S'' := \Next_{S'}\left(T'\right)$
    \item \qquad \qquad \textbf{return} $\Grow_{S''}\left(T'\right)$
    \item \qquad \textbf{else}, \textbf{if} ${\rm Next}(T')$ is defined, \textbf{then}
    \item \qquad \qquad \textbf{return} $\Grow_\epsilon\left({\rm Next}(T')\right)$
    \item \qquad \textbf{else}
    \item \qquad \qquad \textbf{return} the fact that ${\rm Next}(T)$ is not defined
\end{enumerate}
\medbreak
Let us consider the following use of the $\First$ and $\Next$ functions in order to generate
all NLTs of size $n = 3$.
\begin{enumerate}
    \item The first generated NLT is, by definition of the first/next approach, 
    
    $\First(3) =
    \{\epsilon, 0, 1, 2, 01, 12, 012\} = T_1$.

    \item By definition of the first/next approach, the next generated NLT is $\Next(T_1)$.
    By following the Next algorithm, $S = 2$, $S' = \epsilon$, and $T' = \{\epsilon, 0, 1,
    01\}$. We have moreover $\Next_{S'}(T') = 0 = S''$. The generated NLT is
    $\Grow_{S''}(T') = \{\epsilon, 0, 1, 01, 02, 012\} = T_2$.

    \item By definition of the first/next approach, the next generated NLT is $\Next(T_2)$.
    By following the Next algorithm, $S = 02$, $S' = 0$, and $T' = \{\epsilon, 0, 1,
    01\}$. We have moreover $\Next_{S'}(T') = 1 = S''$. The generated NLT is
    $\Grow_{S''}(T') = \{\epsilon, 0, 1, 01, 12, 012\} = T_3$.

    \item By definition of the first/next approach, the next generated NLT is $\Next(T_3)$.
    By following the Next algorithm, $S = 12$, $S' = 1$, and $T' = \{\epsilon, 0, 1,
    01\}$. We have moreover $\Next_{S'}(T') = 01 = S''$. The generated NLT is
    $\Grow_{S''}(T') = \{\epsilon, 0, 1, 01, 012\} = T_4$.

    \item By definition of the first/next approach, the next generated NLT is $\Next(T_4)$.
    By following the Next algorithm, $S = 012$, $S' = 01$, and $T' = \{\epsilon, 0, 1,
    01\}$. We have moreover that $\Next_{S'}(T')$ is not defined, so that the generated NLT
    is $\Grow_\epsilon(\Next(T'))$. Since, as can be easily checked, $\Next(T') =
    \{\epsilon, 0, 01\}$, this next generated NLT is $\Grow_\epsilon(\Next(T')) =
    \{\epsilon, 0, 2, 01, 02, 012\} = T_5$.

    \item By definition of the first/next approach, the next generated NLT is $\Next(T_5)$.
    By following the Next algorithm, $S = 2$, $S' = \epsilon$, and $T' = \{\epsilon, 0,
    01\}$. We have moreover $\Next_{S'}(T') = 0 = S''$. The generated NLT is
    $\Grow_{S''}(T') = \{\epsilon, 0, 01, 02, 012\} = T_6$.

    \item By definition of the first/next approach, the next generated NLT is $\Next(T_6)$.
    By following the Next algorithm, $S = 02$, $S' = 0$, and $T' = \{\epsilon, 0,
    01\}$. We have moreover $\Next_{S'}(T') = 01 = S''$. The generated NLT is
    $\Grow_{S''}(T') = \{\epsilon, 0, 01, 012\} = T_7$.

    \item By definition of the first/next approach, the next generated NLT is $\Next(T_7)$.
    By following the Next algorithm, $S = 012$, $S' = 01$, and $T' = \{\epsilon, 0, 01\}$.
    We have moreover that $\Next_{S'}(T')$ is not defined and that, as can be easily
    checked, that $\Next(T')$ is not defined neither. Therefore, $\Next(T_7)$ is not defined
    and the generation of the NLT of size $3$ is done.
\end{enumerate}
\medbreak
\begin{theorem}
For any $n \in \N$, the set of NLTs obtained by iteratively applying $\Next$ on $\First(n)$
is the set $\SetNLT(n)$.
\end{theorem}
\begin{proof} We proceed by induction on $n$. When $n = 0$ or $n = 1$, the property
holds immediately since the algorithm Next, applied on $\First(n)$, produces no output. This
is in accordance with the fact that $\First(n)$ is the only NLT of size $n$.

Assume now that the property holds for an $n - 1 \geq 0$. Let us look at what the algorithm
Next returns on the input of an NLT $T$ of size $n \geq 2$. At line 2, the set $S$ is
well-defined since, as $T$ is a topology, the intersection of all sets of $T$ containing $n
- 1$ belongs to $T$. The NLT $T'$ computed at line 4 is the NLT of size $n - 1$ obtained by
deleting the element $n$ of $T$. At line 5, the algorithm tests if $S' = S \setminus \{n -
1\}$ is not the last set of $T'$ with respect to the length-lexicographic order. When this
is the case, $S''$ is defined as the set following $S'$ (with respect to the
length-lexicographic order) in $T'$. Again in this case, the algorithm returns the NLT
$T''$ obtained by doubling the interval $[S'', X_n]$. By
Theorem~\ref{thm:increase_topology}, $T''$ is the NLT of a poset $P''$ obtained by adding
a new element $n$ to the poset $P'$ whose $T'$ is its NLT. In this way, by applying
iteratively $\Next$ on $T$, this process generates all NLTs obtained by adding $n$ to $T'$
in all possible ways. It remains to explain what happens when $S'$ is the last set of $T'$
with respect to the length-lexicographic order. In this case, the algorithm computes
recursively $T'' = \Next(T')$. By induction hypothesis, $T''$ is the next generated
element of $T'$. The algorithm returns then $\Grow_\epsilon(T'')$, which is the NLT $T'''$
obtained by doubling $T''$ in its entirety. This NLT $T'''$ is the NLT of a poset $P'''$
obtained by adding a new element $n$ to the poset $P''$ whose $T''$ is the NLT in such a
way that $n$ is the greatest element of $P'''$. When the algorithm Next is again executed
on $T'''$, observe that $S' = \emptyset$, so that $T'''$ is, as expected, the first next
NLT of $T$. In this way, all NLTs of size $n$ are generated iteratively by applying the
algorithm Next. Finally, when $S'$ is the last set of $T'$ and $\Next(T')$ is not defined,
$T$ is already the last generated element of size $n$. In this case, the algorithm returns
that $\Next(T)$ is not defined.
\end{proof}

\end{document}